\documentclass[11pt]{amsart}

\usepackage{amssymb}

\usepackage{tikzexternal}
{dgsarxiv}    

\usepackage{hyperref} 

\numberwithin{equation}{section} 
\theoremstyle{plain}
\newtheorem{cor}[equation]{Corollary}
\newtheorem{lem}[equation]{Lemma}
\newtheorem{prop}[equation]{Proposition}
\newtheorem{thm}[equation]{Theorem}

\theoremstyle{definition}

\newtheorem{jsa}[equation]{Jump set algorithm}

\newtheorem{dfn}[equation]{Definition}

\newtheorem{exa}[equation]{Example}

\newtheorem{rem}[equation]{Remark}

\newcommand{\cref}[1]{Corollary~\ref{#1}}   

\newcommand{\im}{\operatorname{Im}}

\newcommand{\Z}{\mathbb Z}  
\newcommand{\R}{\mathbb R} 
\newcommand{\nnR}{\mathbb R_{\geq 0}} 

\newcommand{\inner}{\ensuremath{\langle \cdot , \cdot  \rangle }}

\renewcommand{\phi}{\varphi}
\renewcommand{\emptyset}{\varnothing}
\newcommand{\si}{\sigma}
\newcommand{\al}{\alpha}

\newcommand{\tM}{\widetilde{M}}
\newcommand{\tX}{\widetilde{X}}

\newcommand{\tg}{\tilde{g}}

\newcommand{\mM}{\mathcal{M}}
\newcommand{\mL}{\mathcal{L}}
\newcommand{\mF}{\mathcal{F}}

\newcommand{\CovSpec}{\operatorname{CovSpec}} 
\newcommand{\Jump}{\operatorname{Jump}} %
\newcommand{\Hom}{\operatorname{Hom}} %
\newcommand{\Gl}{\operatorname{Gl}} %
\newcommand{\Fil}{\operatorname{Fil}} %
\newcommand{\fund}{\pi_1(M)} 

\newcommand{\Image}{\operatorname{Im}}

\newcommand{\diam}{\operatorname{diam}}

\newcommand{\eps}{{\epsilon}}

\begin{document}

\newcommand{\spacing}[1]{\renewcommand{\baselinestretch}{#1}\large\normalsize}
\spacing{1.14}

\title{Sunada's method and the covering spectrum}

\author[B. de Smit]{Bart de Smit $^\ast$}
\address{Universiteit Leiden\\Mathematisch Instituut\\Postbus 9512 \\2300 RA
Leiden\\} \email{desmit@math.leidenuniv.nl}
\thanks{$^\ast$ Funded in part by the European Commission under contract
MRTN-CT-2006-035495}

\author[R. Gornet]{Ruth Gornet $^\dagger$}
\address{University of Texas \\ Department of Mathematics \\ Arlington, TX
76019\\} \email{rgornet@uta.edu}
\thanks{$^\dagger$ Research partially supported by NSF grant DMS-0204648}

\author[C. J. Sutton]{Craig J. Sutton$^\sharp$}
\address{Dartmouth College\\ Department of Mathematics \\ Hanover, NH 03755\\}
\email{craig.j.sutton@dartmouth.edu}
\thanks{$^\sharp$ Research partially supported by an NSF Postdoctoral Research
Fellowship, NSF grant DMS-0605247 and a Woodrow Wilson Foundation Career Enhancement Fellowship}




\begin{abstract}
In 2004, Sormani and Wei introduced the covering spectrum: a geometric
invariant that isolates part of the length spectrum of a Riemannian manifold.
In their paper they observed that certain Sunada isospectral manifolds share
the same covering spectrum, thus raising the question of whether the covering
spectrum is a spectral invariant.
In the present paper we describe a group theoretic condition under which 
Sunada's method gives manifolds with identical covering spectra. When the
group theoretic condition of our method is not met, we are able to construct
Sunada isospectral manifolds with distinct covering spectra in dimension 3
and higher. Hence, the covering spectrum is not a spectral invariant. The main
geometric ingredient of the proof has an interpretation as the
minimum-marked-length-spectrum analogue of Colin de Verdi\`{e}re's classical
result on constructing metrics where the first $k$ eigenvalues of the Laplace
spectrum have been prescribed.
\end{abstract}

\maketitle


\section{Introduction}

Two widely studied geometric invariants of a closed connected Riemannian
manifold $(M,g)$ are the Laplace spectrum and the length spectrum. The
\emph{Laplace spectrum} (or \emph{spectrum}) is the non-decreasing sequence of
eigenvalues, considered with multiplicities, of the Laplace operator acting on
the space of smooth functions of $M$. While the spectrum is known to encode
some geometric information such as the dimension, volume and total scalar
curvature, it is by now well established that the spectrum does not uniquely
determine the geometry of a Riemannian manifold (e.g., \cite{Milnor},
\cite{Vigneras}, \cite{Gordon2}, \cite{Szabo}, \cite{Schueth}, \cite{Sut}). In
particular, in 1985 Sunada devised a method that allows one to construct an
abundance of isospectral manifolds by exploiting certain finite group actions
\cite{Sunada}, which were originally studied by Gassmann \cite{Gassmann}. The
\emph{length spectrum} is the collection of lengths of the smoothly closed
geodesics in
$(M,g)$, where the multiplicity of a length is counted according to the number
of free homotopy classes containing a geodesic of that length. If we ignore
multiplicities, the resulting set of non-negative numbers is known as the
\emph{weak} (or \emph{absolute}) \emph{length spectrum}. As with the Laplace
spectrum, it is known that the 
length spectrum does not uniquely
characterize the geometry of a manifold.

A classical pursuit in geometry, dynamics and mathematical physics is to
understand the mutual influences of the Laplace and  
length spectra of a
Riemannian manifold. By using the Poisson summation formula one can show that
any two flat tori are isospectral if and only if they share the same length
spectrum. The work of Colin de Verdi\`{e}re shows that for a generic Riemannian
manifold---i.e., a manifold equipped with a ``bumpy'' metric 
\cite{Abraham}---the weak length spectrum is determined by its Laplace spectrum
\cite{CdV}. Furthermore, Chazarain \cite{Chazarain} and Duistermaat and
Guillemin \cite[Cor. 1.2]{DuGu} demonstrated that for an arbitrary Riemannian manifold,
the \emph{weak} length spectrum contains the singular support of the trace of
its wave group (a spectrally determined tempered distribution). Hence, the
Laplace spectrum always determines a non-trivial subset of the weak length
spectrum of a Riemannian manifold. However, the interesting and important
question of whether the weak length spectrum is actually a spectral invariant
remains open.

In \cite{SW} Sormani and Wei introduced the covering spectrum: a geometric
invariant that is related to the length spectrum of a Riemannian manifold $(M,g)$, and that
``roughly measures the size of the one dimensional holes in the space''.
In \cite{SW} the covering spectrum is
computed by considering a certain family $\{\tM^{\delta} \}_{\delta >0}$ of
regular coverings of $M$, where $\tM^\delta$ covers $\tM^{\epsilon}$ for
$\epsilon > \delta$, and selecting the values of $\delta$ where the isomorphism
type of the cover changes. That is, we look for ``jumps'' in the ``step
function'' $\delta \mapsto \tM^{\delta}$. When viewed within this framework the
definition actually applies to all complete length spaces, and Sormani and Wei
demonstrated that the covering spectrum is well-behaved under Gromov-Hausdorff
convergence. 

In this paper we present a slightly different, yet compatible, definition of
the covering spectrum that is applicable to any metric space (see
Section~\ref{Sec:CovSpec}). 
In short, we form the covering spectrum of $M$ by assigning a non-negative real
number $r(N/M)$ to \emph{every} non-trivial covering $N$ of $M$. In the case
that $(M,g)$ is a compact
Riemannian manifold this number is half the length of the shortest closed geodesic
that has a lift to $N$ that is not a closed loop; see Corollary
\ref{cor:geolength} (cf. \cite[Lemma 4.9]{SW}).
For example, if $\widetilde M$ is the universal cover of $M$,
then $r(\widetilde M/M)$ is half of the systole of $M$.
The covering spectrum of $(M,g)$,
denoted $\CovSpec(M,g)$, is then defined to be the collection of all $r(N/M)$
as $N$ ranges over all non-trivial covers of $M$. 
Thus, $2\CovSpec(M,g)$ is the portion of the weak length spectrum
consisting of those lengths 
that are ``seen'' by some covering space
as the length of the shortest closed geodesic having a non-closed lift.

As an example, consider the flat $3\times 2$ torus $T^2 = S^1(3) \times
S^1(2)$, where $S^1(c)=\R/c\Z$ denotes the circle of circumference $c$. 
Then one can easily verify that $\CovSpec(T^2) = \{1, \frac{3}{2} \}$ \cite[Exa.
2.5]{SW}. In this case we see that the covering spectrum consists entirely of
half the successive minima of the corresponding lattice $3\Z\times 2\Z$.
However, as we show in Example~\ref{Exa:FlatTori} this is not the case for
all flat tori.

Having identified the covering spectrum as a geometrically
determined finite part of the weak length spectrum, one 
may wonder about the mutual influences
between the covering spectrum and Laplace spectrum. Along these lines, Sormani
and Wei found that so-called Komatsu pairs of Sunada isospectral manifolds
share the same covering spectrum \cite[Ex. 10.5]{SW}. However, as we will show
in Section~\ref{Sec:Isospectral}, their claim \cite[Example 10.3]{SW} that
certain pairs of isospectral Heisenberg manifolds due to Gordon have distinct
covering spectra is false \cite{SWerrata}, thus keeping alive the question of
whether the covering spectrum is a spectral invariant.

By engaging in both a geometric and a group theoretic analysis of the covering
spectrum and its relation to the Sunada condition, we provide a negative answer
to this question. Specifically, we show that in dimension $3$ and higher there
are Sunada isospectral
manifolds with distinct covering spectra. In dimension 4 these include certain
isospectral flat tori due to Conway and Sloane.

In closing, we briefly return to the relationship between the weak length
spectrum and the Laplace spectrum. As we noted above, the Laplace spectrum of a
manifold always determines a particular subset of its weak length spectrum, and
under certain genericity conditions it is known that the Laplace spectrum
completely determines the weak length spectrum. However, in general, this
relationship is not fully understood. In Corollary~8.8 of \cite{SW}, Sormani
and Wei make the interesting observation that a continuous family of manifolds
sharing the same \emph{discrete} weak length spectrum will necessarily share
the same covering spectrum. That is, the covering spectrum is an invariant of a
particular flavor of \emph{iso-length-spectral} deformation. In light of the
fact that we have shown that the covering spectrum is not a spectral invariant
it would appear to be of interest to explore the covering spectrum of certain
continuous families of \emph{isospectral} manifolds in the literature.


\section{Main results and overview}\label{Sec:Results}
\noindent
In this section, we state our main results and give an overview of the
material covered in later sections. We first recall Sunada's construction.  

Let $G$ be a finite group, and let $H$
and $H'$ be subgroups. Suppose that $G$ acts on a closed connected manifold
$M$, and suppose that the action is free in the sense that every non-trivial
element of $G$ acts without fixed points.  We then consider the quotient
manifolds $H\backslash M$ and $H'\backslash M$, which both cover $G\backslash
M$, as indicated in Figure~\ref{fig:SunadaDiagram}, where
the labels at regular coverings are their deck transformation groups.

\medskip
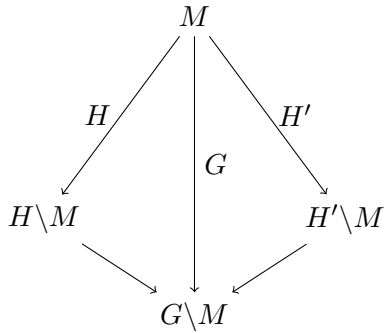
\begin{figure}
\begin{center}
\begin{tikzpicture}
\node (N) at (0,4) {$M$};
\node (M) at (-2,1.3) {$H\backslash M$};
\node (MM) at (2,1.3) {$H'\backslash M$};
\node (GM) at (0,0) {$G\backslash M$};
\draw [->] (N) -- node[left] {$H$} (M);
\draw [->] (N) -- node[right] {$H'$}(MM);
\draw [->] (N) -- node[right] {$G$}(GM);
\draw [->] (M) -- (GM);
\draw [->] (MM) -- (GM);
\end{tikzpicture}
\end{center}
\caption{Diagram of Covering Spaces}\label{fig:SunadaDiagram}
\end{figure}
\begin{thm}[Sunada \cite{Sunada}, Pesce \cite{Pesce4}]\label{thm:Sunada}
The following are equivalent:
\begin{enumerate}
\item for every conjugacy class $C$ of $G$ we have $\#(H\cap C) = \#(H'\cap
C)$;
\item for every $G$-invariant Riemannian metric on $M$, the quotient Riemannian
manifolds $H\backslash M$ and $H'\backslash M$ have the same Laplace spectrum.
\end{enumerate}
\end{thm}

The fact that $(1)$ implies $(2)$ is Sunada's theorem, and it is the only
implication that will be used in this paper (see \cite{Brooks},
\cite[Sec.  3]{GMW} for alternative proofs and comments).  Many of the examples
of isospectral manifolds in the literature can be understood within the
framework  of Sunada's theorem and its generalizations (e.g., \cite{DG},
\cite{Berard1}, \cite{Berard2}, \cite{Pesce}, \cite{Sut} and \cite{Ballmann}).
The group theoretic condition $(1)$  
is known by several names in
the literature: Perlis \cite{Perlis} says that $H$ and $H'$ are Gassmann
equivalent after Gassmann \cite{Gassmann}, who first used this condition in
1926, and spectral geometers (e.g., \cite{Brooks}) frequently say that $H$ and
$H'$ satisfy the Sunada condition. Others say that $H$ and $H'$ are almost
conjugate \cite{Buser, GK}, or linearly equivalent \cite{DL} or that $H$ and
$H'$ induce the same permutation representation \cite{GW}. In this paper we
will express that condition (1) holds by saying that $H$ and $H'$ are
\emph{Gassmann-Sunada} equivalent, or that
$(G,H,H')$ is a \emph{Gassmann-Sunada triple}. 

The statement that $(2)$ implies
$(1)$ in  
Theorem \ref{thm:Sunada}
follows from the work of Pesce \cite{Pesce4}, and it shows that the
Gassmann-Sunada condition
is not only sufficient, but also necessary if we want it to ensure that
$H\backslash M$ and $H'\backslash M$ are isospectral for all possible choices
of compatible Riemannian metrics in the diagram above.

In Section \ref{Sec:Proof} we establish the following analogue of Sunada's
method for the covering spectrum.

\begin{thm}\label{thm:CovSunada}
Let $G, H, H'$ and  $M$ be as above.
If $M$ is simply connected and of dimension at least $3$ then,
the following are equivalent:
\begin{enumerate}
\item[(3)]
for all subsets $S, T$ of $G$ that are stable under conjugation we have
$$
    \langle H\cap S\rangle = \langle H\cap T\rangle
    \iff
    \langle H'\cap S\rangle = \langle H'\cap T\rangle;
$$
\item[(4)] for every $G$-invariant Riemannian metric on $M$, the quotient
Riemannian manifolds $H\backslash M$ and $H'\backslash M$ have the same
covering spectrum.
\end{enumerate} More generally, if $M$ is simply connected,
then $(3)$ implies~$(4)$, and if $M$ has dimension at least $3$, then $(4)$
implies $(3)$.
\end{thm}

In the last statement of Theorem~\ref{thm:CovSunada}, the condition that $M$ is
simply connected cannot be omitted. As will become evident after the discussion in 
Section~\ref{Sec:Proof}\label{Explanation}, this follows from the fact that 
if $N$ is a normal subgroup of $G$ that is contained in 
$H \cap H'$, then the triple $(G/N, H/N, H'/N)$ can satisfy condition $(3)$ while
the triple $(G, H, H')$ need not satisfy this condition 
(see Remark~\ref{rem:ReducedTriples} for an example).
We do not know if $(4)$ implies $(3)$ for
all $2$-dimensional manifolds.  In order to understand the relevance of
condition $(3)$ in this result, and to explain the ingredients of the proof, we
introduce the purely algebraic concept of a \emph{length map} on a group.

\begin{dfn}\label{dfn:lengthmap}
A \emph{length map} on a group $G$ is a map $m\colon\; G \to \nnR$
from $G$ to the set of non-negative real numbers that satisfies
\begin{enumerate}
\item[$(i)$] $m(g)>0$ for all $g\in G-\{1\}$;
\item[$(ii)$] $m(g)=m(hgh^{-1})$ for all $g,h\in G$;
\item[$(iii)$] $m(g^k)\le |k|m(g)$ for all $k\in \Z$ and $g \in G$.
\end{enumerate}
\end{dfn}

\noindent
Taking $k=0$ and $k=-1$ we see that $(iii)$ implies that $m(1)=0$ and
$m(g^{-1})=m(g)$ for all $g\in G$. 
The motivation for the preceding definition is found in the following example.

\begin{exa}[Minimum marked length map, {\cite[Def.~4.5]{SW}}]\label{exa:MinMark}
Let $(M,g)$ be a \emph{compact} Riemannian manifold, and let
$\pi_1(M)$ be its fundamental group with respect to some base point.
The \emph{minimum marked length map} $m_g\colon\; \pi_1(M) \to \nnR$ assigns to
each $\gamma \in \pi_1(M)$ the length of the shortest closed geodesic in the
free homotopy class determined by $\gamma$.  
Equivalently, we can set $m_g(\gamma) = \min_{x \in
\tM} d(x, \gamma \cdot x)$, where $(\tM, \tg)$ is the universal Riemannian
cover of $(M,g)$, $d$ is the metric structure induced by $\tg$ and (after a
choice of base point) the group $\fund$ acts on $\tM$ via deck transformations
(see \cite[Sec.  6]{Spanier}). 
We refer the reader to \cite[Lem. 4.4]{SW} for the reason why each free
homotopy class in the definition above has a shortest closed geodesic.  Since
the conjugacy classes of $\pi_1(M)$ naturally correspond to the free homotopy
classes of loops in $M$, we see that $m_g$ is a length map.  Since $M$ is
compact, the image of $m_g$ is closed and discrete in $\nnR$
\cite[Lem.~4.6]{SW}.
\end{exa}

\begin{dfn}\label{def:IntroJumpSet}
Given a length map $m$ on a group $G$, we define for each $\delta > 0$ the subgroup $\Fil_m^\delta G
=\langle g\in G: m(g)< \delta\rangle$ of $G$. When $\delta <\eps$, we have
$\Fil_m^\delta G \subset \Fil_m^\eps G$, so this defines an
increasing filtration, denoted $\Fil_m^\bullet G$, of $G$ by normal subgroups
indexed by the positive real numbers.  We define its \emph{jump set}
$\Jump(\Fil_m^\bullet G)$ to be the set of all $\delta > 0$ such that for all
$\eps>\delta$ we have $\Fil_m^\delta G\ne \Fil_m^\eps G$.  
\end{dfn}

We will see that
this notion of the jump set of the filtration defined by a length map provides
the connection between conditions $(3)$ and $(4)$ of Theorem
\ref{thm:CovSunada}.  More specifically, in Section \ref{Sec:CovFund} we will
show the following.

\begin{prop}\label{prop:CovJump}
The covering spectrum of a compact Riemannian manifold $(M,g)$ is given by
\[
\CovSpec(M,g)= \frac{1}{2} \Jump(\Fil^\bullet_{m_g}\pi_1(M)).
\]
\end{prop}

In Lemma \ref{lemma:jumptriple} we will see that condition $(3)$ of
Theorem~\ref{thm:CovSunada} holds if and only if for every length map $m$ on
$G$ the restrictions to $H$ and to $H'$ have identical jump sets. When $(3)$
holds, we will say that $H$ and $H'$ are \emph{jump equivalent} or that
$(G,H,H')$ is a \emph{jump triple}. Combining Lemma \ref{lemma:jumptriple} with
Proposition \ref{prop:CovJump} it is not hard to show that $(3)$ implies $(4)$
for simply connected~$M$. In order to show the other, more difficult,
implication of Theorem~\ref{thm:CovSunada}, we will assume that $(3)$ does not
hold and then proceed to construct a Riemannian metric on $M$ so that the
resulting minimum marked length map has a special property.  In order to
construct this metric it will be necessary to understand the extent to which
length maps actually arise from minimum marked length maps of Riemannian
manifolds. That is, we would like to know which length maps on the fundamental
group of a given manifold $M$ arise as the minimum marked length map associated
to some Riemannian metric on $M$. 

In Section~\ref{Sec:Prescribing} we take up this line of inquiry and we establish
Theorem~\ref{thm:SystoleLengthMap}, which can be summarized as follows.

\begin{thm}
\label{thm:IntroSystole}
Let $M$ be a closed manifold of dimension at least three and let $S \subset
\fund$ be a finite set.  If $m: S \to \nnR$ is the restriction to $S$ of a
length map on $\fund$, then there exists a Riemannian metric $g$ on $M$ such
that $m_g(s) = m(s)$ for any $s \in S$. That is, $m$ can be extended to the
minimum marked length map associated to some metric $g$ on $M$.
\end{thm}

\noindent
Hence, in the case where $M$ is a closed manifold with finite fundamental group
and dimension at least $3$, the above demonstrates that there is no difference
between length maps on $\pi_1(M)$ and the minimum marked length maps associated
to $M$.

In Theorem~\ref{thm:SystoleLengthMap}---the detailed statement of the above---we 
actually prove that we have some control over the extension of $m$.
More specifically, we demonstrate that we are able to prescribe 
which free homotopy classes contain the shortest 
closed geodesics (see Remark~\ref{rem:SystoleLengthMap}).
In Section~\ref{Sec:Proof}, we will 
use this to prove
Theorem~\ref{thm:CovSunada}.

Theorem~\ref{thm:SystoleLengthMap} also provides a way to characterize 
the initial segments of the so-called minimum marked length spectrum
of a Riemannian metric on a closed manifold of dimension at least three.
Since this is perhaps
of independent interest we will formulate this after the following definition.

\begin{dfn}[cf.~\cite{GornetMast}]\label{def:MinLSpec}
Let $(M,g)$ be a Riemannian manifold with associated minimum marked length map
$m_g: \fund \to \nnR$.  The value $m_g([\si])$ depends only on the class of the
loop $\si$ in the set of unoriented free homotopy classes $\mF(M)$ of loops in
$M$; whereby the \emph{unoriented free homotopy class} of $\sigma$ we mean the
collection of all loops freely homotopic to $\sigma$ or its inverse
$\bar{\sigma}$.  Hence, we obtain an induced map $\mF(M)\to \nnR$, which is
again denoted by $m_g$, and in this incarnation it is known as the
\emph{minimum marked length spectrum}.  Alternatively, one can define the
minimum marked length spectrum as the set of ordered pairs $(m_g(c),c)$ where
each length $m_g(c)$ is ``marked'' by the unoriented free homotopy class $c\in
\mF(M)$. This set of pairs is a subset of the marked length spectrum as defined
in \cite{Gornet}, for example.
\end{dfn}

\newcommand\Fvar{{C}}
\newcommand\fvar{{c}}
\newcommand\lvar{{l}}
\begin{thm}\label{thm:varyg}
Suppose that $M$ is a closed connected manifold of dimension at least three.
Let $\Fvar = ( \fvar_1, \fvar_2, \ldots, \fvar_k )$ be 
a sequence of
distinct elements of $\mF(M)$ where the first element $\fvar_1$ is trivial.
Then for every sequence $0 = \lvar_1 < \lvar_2 \leq \cdots \leq \lvar_k$
of real numbers the following are equivalent:
\begin{enumerate}
\item[(5)] the sequence $\lvar_1,\ldots, l_k$ is $\Fvar$-admissible 
(see Definition~\ref{dfn:Admissible} and Example~\ref{exa:Admissible});
\item[(6)] there is a Riemannian metric $g$ on $M$ such that the minimum marked
length map $m_g$ 
satisfies $m_g(\fvar_i) = \lvar_i$ for all $i$ and
$m_g(\fvar) \geq \lvar_k$ for all $c\in \mF(M)-\{\fvar_1, \ldots , \fvar_k \}$.
\end{enumerate}
In particular, there is a metric $g$ on $M$ such that the systole is achieved
in the unoriented free homotopy class $\fvar_2$. 
\end{thm}

The statement that $(5)$ implies $(6)$ in the above is  
a
minimum-marked-length-spectrum analog of a classical
theorem due to Colin de Verdi\`{e}re \cite{CdV2}, which states that given a
connected closed manifold $M$ of dimension at least $3$ and a finite sequence
$a_1  = 0 < a_2\leq a_3 \leq \cdots \leq a_k$ there is a Riemannian metric $g$
on $M$ such that the sequence gives exactly the first $k$ eigenvalues, counting
multiplicities, of the associated Laplacian. 
However, unlike Colin de Verdi\`{e}re's result, the sequence 
$0= \lvar_1 < \lvar_2 \leq \lvar_3 \cdots \leq \lvar_k$ in the above cannot be
chosen arbitrarily:  
it will depend on our choice of the sequence
$\Fvar = (\fvar_1, \fvar_2, \ldots , \fvar_k)$ in $\mF(M)$.
The above theorem then tells us that given such a choice for $\Fvar$, 
the $\Fvar$-admissibility of
$0= \lvar_1 < \lvar_2 \leq \lvar_3 \cdots \leq \lvar_k$ is a necessary and
sufficient condition for the existence of a metric $g$ such that the $i$th
smallest value of the minimum marked length spectrum is $m_g(\fvar_i)=l_i$ for
$i=1, \ldots, k$.

With Theorem~\ref{thm:SystoleLengthMap} in place, we will complete the proof of
Theorem \ref{thm:CovSunada} in Section~\ref{Sec:Proof}.  In Section
\ref{Sec:Groups} we present some group theoretic context and results concerning
Gassmann-Sunada triples (condition $(1)$ of Theorem~\ref{thm:Sunada}) and jump
triples (condition $(3)$ of Theorem~\ref{thm:CovSunada}).
There are some key differences in the behavior of jump triples and
Gassmann-Sunada triples. For instance, we have $[G:H]=[G:H']$ for every
Gassmann-Sunada triple $(G,H,H')$, but not for every jump triple. The behavior
with respect to dividing out normal subgroups is also different: if $N$ is a
normal subgroup of $G$ that is contained in $H\cap H'$, then
$H$ and $H'$ are Gassmann-Sunada equivalent in $G$
if and only if $H/N$ and $H'/N$ are Gassmann-Sunada equivalent in $G/N$.
However, as we noted earlier, $H/N$ and $H'/N$ can be jump equivalent in $G/N$ 
even when $H$ and $H'$ are not jump equivalent in $G$ (see Remark~\ref{rem:ReducedTriples}).

While most small examples of Gassmann-Sunada triples turn out to be jump
triples (see Example \ref{exa:small}) we will show how to adapt some well-known
constructions of Gassmann-Sunada triples to obtain Gassmann-Sunada triples that
are not jump triples.

Finally, our analysis leads us to the following conclusion in Section \ref{Sec:Isospectral}.

\begin{thm}\label{thm:distinct}
For each $n\ge 3$ there are closed Riemannian manifolds of dimension $n$ with
identical Laplace spectra and distinct covering spectra.
\end{thm}

\noindent
As we will see in Example~\ref{exa:tori}, these will include certain
isospectral flat tori of dimension $4$ due to Conway and Sloane \cite{CS}.
In a forthcoming paper we employ different techniques to 
construct isospectral surfaces with distinct covering spectra \cite{DGS}.

We conclude the paper with a study of the covering spectrum of
Heisenberg manifolds that refutes 
\cite[Example 10.3]{SW}.


\section{The covering spectrum of a metric space}\label{Sec:CovSpec}

The notion of the covering spectrum was defined by Sormani-Wei \cite{SW} for
complete length spaces.  In this section we provide an alternate definition
that works for any metric space.  For compact length spaces---the main focus of
this paper---our definition coincides with the definition given by Sormani and
Wei, and for non-compact complete length spaces the only difference is that $0$
is sometimes an element of our covering spectrum.  In order to present this
definition, we first review some terminology from covering space theory.

Let $X$ be a topological space. By a \emph{space over} $X$ we mean a
topological space $Y$ along with a continuous map $p\colon\; Y\to X$.  For such
a space $Y$ over $X$ and $Z\subset X$ we say that $Y$ is \emph{trivial over}
$Z$, or it \emph{evenly covers} $Z$, if $p^{-1}(Z)$ is a disjoint union of
open subspaces of $p^{-1}(Z)$ that are each mapped homeomorphically onto $Z$
via the map~$p$.  A \emph{covering space} of $X$ is a space $Y$ over $X$ that
is locally trivial; that is, each point $x\in X$ has an open neighborhood $U$
so that $Y$ is trivial over~$U$. We say that a covering space $Y$ of
$X$ is \emph{trivial} if it is trivial over $X$.

\begin{dfn}\label{dfn:covrad}
Let $(X,d)$ be a metric space and let $p\colon\; Y \to X$ be a covering space
of~$X$.  For $\delta\in \R_{\ge 0} \cup \{\infty\}$ we say that $Y$ is
\emph{$\delta$-trivial} over $X$ if for each $x\in X$ the cover $Y$ is trivial
over the open ball of radius $\delta$ centered at~$x$.  The \emph{covering
radius} $r(Y/X) \in \R_{\ge 0} \cup \{\infty\}$ of $Y$ over $X$ is given by
$$
r(Y/X)=\sup\;\{\delta \colon\; Y \textrm{ is }\delta\textrm{-trivial over }X\}.
$$
\end{dfn}

Clearly, we have $r(Y/X)=\infty$ for any trivial cover $Y\to X$,
and $r(Y/X) \le \diam(X)$ for any non-trivial cover $Y\to X$.

\begin{rem} 
If $X$ is a compact metric space, then the Lebesgue
number lemma tells us that the covering radius of any cover will be
non-zero.  However, in the next example we see that non-compact metric spaces
can admit covers that have zero covering radius.
\end{rem}

\begin{exa}[Zero covering radius]\label{exa:ZeroRadius}
Let $X$
be the metric \emph{subspace} of the Euclidean plane equipped with the standard
metric that one gets by
connecting a sequence of circles with radius going to zero as in Figure~\ref{fig:ZeroCovRad}.
Then for any non-trivial covering space $Y$ of $X$ the covering radius $r(Y/X)$
is the infimum of the radii of those circles that are not evenly covered.
This space $X$ has a simply connected universal cover, which therefore
has covering radius zero over~$X$.

\begin{figure}
\begin{center}
\begin{tikzpicture}[scale=1.4]
\draw (0,0) grid (4.5,0.999);
\draw[fill=white] (0,1) circle (0.5cm);
\draw[fill=white] (1,1) circle (0.3cm);
\draw[fill=white] (2,1) circle (0.18cm);
\draw[fill=white] (3,1) circle (0.108cm);
\draw[fill=white] (4,1) circle (0.0648cm);
\draw[fill=black] (4.7,0.2) circle (0.02cm);
\draw[fill=black] (4.9,0.2) circle (0.02cm);
\draw[fill=black] (5.1,0.2) circle (0.02cm);
\end{tikzpicture}
\end{center}
\caption{A metric space with zero covering radius}\label{fig:ZeroCovRad}
\end{figure}
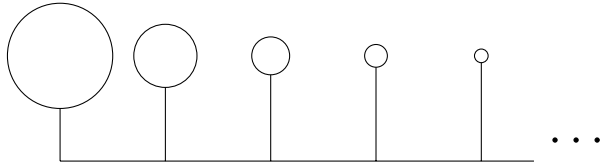
\end{exa}

\begin{rem}
In Lemma \ref{lem:lpc} below we show that if $(X,d)$ is a locally path
connected space, then $p: Y \to X$ is actually an $r(Y/X)$-trivial cover.
In general this may fail; see Example~\ref{exa:InfRad} below. In
particular, we can have $r(Y/X) = \diam(X)$ without $Y \to X$ being
$\diam(X)$-trivial.
\end{rem}

\begin{exa}[Infinite covering radius]\label{exa:InfRad}
Let us define a metric subspace $X$ of the Euclidean plane that is
locally path connected at all but a single point, together with a non-trivial
double cover that is trivial over every bounded subset of~$X$.  
As is shown in Figure~\ref{fig:InfiniteCovRad}, we take a union
of infinitely many circles, all tangent to each other at the same point $x$,
with the radius of the circles going to infinity, where in each circle we omit
an open interval lying around the point $x$ whose size is shrinking to
zero as we go through the infinite sequence of circles.  Thus, what remains of
each circle is a closed interval, and the end points of these intervals give a
sequence converging to~$x$. Then we also add the single point $x$ to our space.
Informally one might say that $X$ has a loop that only closes on an infinite
scale.
Now take any point $z$ inside the circles that is not in~$X$.
Then the plane with $z$ removed has a connected double cover,
and we restrict this cover to~$X$. We leave it to the reader to check that this
covering has the stated properties.

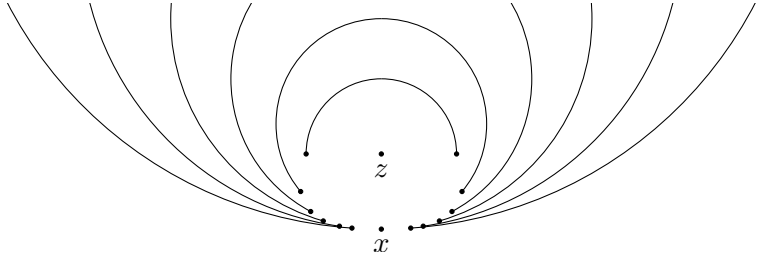
\begin{figure}[h]
\begin{center}
\begin{tikzpicture}
\clip (-6.1,-0.3) rectangle (6.1,3);

\draw (0,2cm) arc (90:0:1cm);
\draw (0,2cm) arc (90:180:1cm);
\draw (0,2.8cm) arc (90:-40:1.4cm);
\draw (0,2.8cm) arc (90:220:1.4cm);
\draw (0,4cm) arc (90:-62:2cm);
\draw (0,4cm) arc (90:242:2cm);
\draw (0,5.6cm) arc (90:-74:2.8cm);
\draw (0,5.6cm) arc (90:254:2.8cm);
\draw (0,8cm) arc (90:-82:4cm);
\draw (0,8cm) arc (90:262:4cm);
\draw (0,11.2cm) arc (90:-86:5.6cm);
\draw (0,11.2cm) arc (90:266:5.6cm);
\draw [fill=black, yshift=1cm] (0:1cm)  circle (0.8pt);
\draw [fill=black, yshift=1cm] (180:1cm)  circle (0.8pt);
\draw [fill=black, yshift=1.4cm] (-40:1.4cm)  circle (0.8pt);
\draw [fill=black, yshift=1.4cm] (220:1.4cm)  circle (0.8pt);
\draw [fill=black, yshift=2cm] (-62:2cm)  circle (0.8pt);
\draw [fill=black, yshift=2cm] (242:2cm)  circle (0.8pt);
\draw [fill=black, yshift=2.8cm] (-74:2.8cm)  circle (0.8pt);
\draw [fill=black, yshift=2.8cm] (254:2.8cm)  circle (0.8pt);
\draw [fill=black, yshift=4cm] (-82:4cm)  circle (0.8pt);
\draw [fill=black, yshift=4cm] (262:4cm)  circle (0.8pt);
\draw [fill=black, yshift=5.6cm] (-86:5.6cm)  circle (0.8pt);
\draw [fill=black, yshift=5.6cm] (266:5.6cm)  circle (0.8pt);

\draw [fill=black] (0,1) circle (0.8pt);
\draw [fill=black] (0,0) circle (0.8pt);
\node[below] at (0,1) {$z$};
\node[below] at (0,0) {$x$};

\end{tikzpicture}
\end{center}
\caption{A metric space with infinite covering radius}\label{fig:InfiniteCovRad}
\end{figure}

One can make this space into a path connected example
by connecting a single end point of each component to the point $x$.
One can also make a similar bounded example: a space within the open disk of
radius 1 in the plane, with a double cover that is $\delta$-trivial for all
$\delta<1$ but not $1$-trivial. 
\end{exa}

\begin{rem}[Non-constant degree]\label{rem:nonconstant}
Given a covering $p\colon\; Y\to X$ the degree of a point $x\in X$ is the
cardinality of the fiber $p^{-1}(x)$ over $x$, which may be infinite.  If this
cardinality is the same for all $x\in X$, then we say that the covering has
constant degree.  Since $Y$ is locally trivial over $X$ this degree is locally
constant on~$X$, so it is constant on the connected components of~$X$.  When
considering the covering radius we will often restrict to the case that the
covering has constant degree.  To see how to express the covering radius in the
general case, note that for any covering $p\colon\;Y\to X$ that is not of
constant degree the fibers of the degree map partitions $X$ into at least two
open subsets $X_i$, where $i$ ranges over a suitable index set $I$, and for
each $i$ the space $Y_i=p^{-1}(X_i)$ is a cover of $X_i$ of constant degree.
Then the covering radius of $Y$ over $X$ is the infimum of the following set:
\[
r(Y/X)=\inf \{d(X_i, X_j)\colon\; i\ne j\}\cup \{r(p^{-1}(X_i)/X_i)\colon\; i
\in I\}.  \]
For example, if $X$ consists of two points and $Y$
is a non-trival covering space of $X$ then $r(Y/X)$ is the distance between the
two points of~$X$.
\end{rem}

\begin{dfn}
The \emph{covering spectrum}, denoted $\CovSpec(X)$, of a metric space $X$
is the set of all $r(Y/X)$ as $Y$ varies over all non-trivial coverings
of $X$ of constant degree.
\end{dfn}

We will show in the next section that for connected locally path connected
spaces the covering spectrum is the jump set of a length map on the fundamental
group. We will also see that for compact Riemannian manifolds,
the prime object of interest in this paper, the definition above is equivalent
to the definition given by Sormani-Wei \cite{SW}; see Remark \ref{rem:SWdef}.
We conclude this section with some general remarks.

The covering spectrum of $X$ is a subset of $\R_{\ge 0}\cup \{\infty\}$.
It is empty if and only if all constant degree covers of $X$ are trivial.
In Example \ref{exa:ZeroRadius} the covering spectrum
is the set of diameters of the circles together with the element~$0$.
If we do not use the Euclidean distance, but instead view $X$ as a length space
by taking the path-length metric, then the covering spectrum consists of 
half the circumferences of the circles together with the element~$0$.
In Example \ref{exa:InfRad} the covering spectrum is~$\{\infty\}$ and that of
its bounded analog is $\{1\}$. The same is true for their path connected
versions. These topological spaces cannot be made into length spaces, so this
is a situation in which the definition of the covering spectrum as given in
\cite{SW} does not apply.

If $X$ is connected then one may drop the condition ``of constant degree'' from
the definition.  Discrete metric spaces have an empty covering spectrum, and
more generally, the covering spectrum of a union of disjoint open subspaces of
a metric space is the union of their covering spectra.


\section{The covering spectrum and the fundamental group}\label{Sec:CovFund}

In this section we see how
to express the covering spectrum in terms of additional structure on the
fundamental group. The main result of this section is a proof of Proposition
\ref{prop:CovJump}. First, we recall some terminology.

A \emph{path} in a topological space $X$
is a continuous map $\sigma\colon\; [0,1]\to X$. We say that $\sigma$ is a
\emph{loop} if $\sigma(0)=\sigma(1)$, and we then say that the loop is
\emph{based at}~$\sigma(0)$.  If $p\colon\; Y\to X$ is a covering map and
$\sigma$ is a
path in $X$, then a \emph{lift} of $\sigma$ to $Y$ is a path $\widetilde
\sigma$ in $Y$ such that $p\circ \widetilde \sigma=\sigma$. The lemma of
\emph{unique
path lifting} states that the map $\widetilde \sigma \mapsto \widetilde
\sigma(0)$ gives a bijection between the set of lifts of $\sigma$ to $Y$ and
the fiber $p^{-1}(\sigma(0))$ over the starting point of~$\sigma$.
Given a path $\sigma$ its \emph{inverse} $\overline \sigma$
is the path given by ${\overline\sigma}(t)=\sigma(1-t)$, and if $\tau$
is also a path in $X$ and $\sigma(1)=\tau(0)$, then $\sigma * \tau$ is the
\emph{composed path} sending $t$ to $\sigma(2t)$ if $t\in [0,1/2]$ and to
$\tau(2t-1)$ when $t\in [1/2,1]$; that is, we travel along $\sigma$ followed by
$\tau$. When $\rho$ is a path in $X$ with $\rho(0)=\tau(1)$, we recall that
$(\sigma * \tau)* \rho$ is path homotopic to $\sigma*(\tau* \rho)$; hence,
there is no ambiguity in writing $[\sigma * \tau * \rho]$ for the path homotopy
class of both paths.
If $\sigma$ is a loop in $X$ based at $x$, then we denote its homotopy class in
the fundamental group $\pi_1(X,x)$ by~$[\sigma]$. The group operation in
$\pi_1(X,x)$ is given by $[\sigma][\tau]=[\sigma*\tau]$ and the inverse by
$[\sigma]^{-1}=[{\overline \sigma}]$.

By a \emph{filtration} $\Fil^\bullet G$ of a group $G$ we mean a family of
subgroups $\Fil^i G$ of $G$, with $i$ ranging over an ordered index set $I$,
such that $\Fil^i G\subset \Fil^j G$ when $i<j$.
The \emph{jump set} of the filtration is the subset of $I$ given by
\[
\Jump(\Fil^\bullet G)=\{i\in I\colon\Fil^iG\ne
\Fil^jG\textrm{ for all }j\in I\textrm{ with } i<j\}.
\]
That is, $i \in \Jump(\Fil^\bullet G)$ if and only if $\Fil^iG$ is a \emph{proper subset} 
of $\Fil^jG$ for all $j >i$ (cf. Definition~\ref{def:IntroJumpSet} and Example~\ref{Exa:FlatTori}).
In this paper our index set $I$ will always be the set of positive real numbers
(with the usual ordering) and the subgroups of the filtration will always be
normal subgroups of~$G$.

For a metric space $(X,d)$ with  
base point $x_0$, we now define a
filtration $\Fil^\bullet \pi_1(X,x_0)$ on the fundamental group $\pi_1(X,x_0)$
of $X$.
This filtration 
will determine the
covering radius of any covering of $X$ and hence encodes the covering spectrum.
For $\delta > 0$ we let $\Fil ^\delta \pi_1(X,x_0)$
be the subgroup of $\pi_1(X,x_0)$ generated by the elements of the form $[\alpha *
\si * \overline{\alpha}]$, where $\si$ is a loop contained inside some open
$\delta$-ball and $\alpha$ is a path from $x_0$ to $\si(0)$.
These generators are precisely the homotopy classes of loops based at $x_0$
that are \emph{freely} homotopic to a loop completely contained in some
open $\delta$-ball. 
We note that in \cite[Section
2]{SW} the same filtration is considered for complete length spaces, where
$\Fil ^\delta \pi_1(X,x_0)$ is denoted by $\pi_1(X, \delta, x_0)$.
The next lemma expresses the covering radius 
(see Definition \ref{dfn:covrad})
in terms of this filtration.

\begin{lem}\label{lem:lpc}
Let $(X,d)$ be a connected and locally path connected
metric space, with base point $x_0$, and let
$Y$ be a non-trivial covering space of~$X$. Then
$Y$ is $r(Y/X)$-trivial over~$X$. Moreover,
$r(Y/X)$ is the maximum of all $\delta > 0$ such that
for every loop $\sigma$ in $X$ based at $x_0$ with
$[\sigma]\in \Fil^\delta\pi_1(X,x_0)$, every lift to $Y$ of
$\sigma$ is a closed loop.
\end{lem}

\begin{proof}
Recall first that for each open subset $U \subset X$ the cover
$Y$ is trivial over $U$ if and only if every lift to $Y$ of
a loop in $U$ is a loop in $Y$; \cite[Lemma 2.4.9]{Spanier}.

To show the first statement, let $B$ be an open ball of radius $r(Y/X)$ in $X$
centered at some point $x \in X$, and let $\sigma$ be a loop whose image is
completely contained inside $B$.  Then since the image of $\sigma$ is compact
we see that there is a $\delta < r(Y/X)$ such that the image of $\sigma$ is
completely contained inside the open ball of radius $\delta$ centered at $x$. By the
definition of the covering radius $r(Y/X)$, the covering $Y \to X$ is
$\delta$-trivial, hence we see that $\sigma$ lifts to a closed loop. It now
follows from our
initial remark that $Y$ is an $r(Y/X)$-trivial cover of $X$.

Next, we note that the connected components of $Y$ are open,
and they are themselves coverings of~$X$. For each component of $Y$ one may
choose a point $y_0$ in the fiber of this component over~$x_0$. With
\cite[Lemma 2.5.11, 2.4.9]{Spanier} one then shows that this component is
$\delta$-trivial over $X$ if and only if $\Fil^\delta\pi_1(X,x_0)$ is contained
in the image of $\pi_1(Y,y_0)$ under the map induced by the covering map.
Thus, the second statement follows.
\end{proof}

\noindent
It follows from the lemma that
the covering spectrum of a locally path connected space does not contain the
element~$\infty$ (cf. Example~\ref{exa:InfRad}).

\begin{prop}\label{prop:CovSpecJSet}
If $(X,d)$ is a connected and locally path connected metric space, then
$$
\CovSpec(X)-\{0\} = \Jump(\Fil^\bullet \pi_1(X,x_0)).
$$
\end{prop}
\begin{proof}
Suppose that $\delta\in\CovSpec(X)-\{0\}$. Then $\delta\ne \infty$ by
Lemma \ref{lem:lpc}, and by definition there is a covering $p\colon\;Y \to X$ with
$r(Y/X)=\delta$.
Again by the lemma, the collection of homotopy classes of
loops based at $x_0$ all of whose
lifts to $Y$ are themselves loops contains $\Fil^\delta \pi_1(X,x_0)$, but it does not
contain the set $\Fil^\epsilon \pi_1(X,x_0)$ for any $\epsilon>\delta$. It follows that
$\delta$ is a jump for the filtration, so the inclusion ``$\subset$'' holds.

Now, let $\delta>0$. Following \cite[Def. 2.3]{SW} we see with
\cite[Theorem 2.5.13]{Spanier} that there is a connected and locally path
connected covering space $p: \tX^\delta \to X$ such that
$p_{\#}(\pi_1(\tX^\delta)) = \Fil^\delta \pi_1(X,x_0) \subset \pi_1(X,x_0)$.
Hence, $\Fil^\delta \pi_1(X,x_0)$  consists of exactly those classes of loops
whose lifts to $\tX^\delta$ are all loops. With the lemma, or by using
\cite[Lemma 2.5.11, 2.4.9]{Spanier}, we see that 
for all $\delta \in \Jump(\Fil^\bullet \pi_1(X,x_0))$,
we have $r(\tX^\delta/X)=\delta$ and ``$\supset$'' follows.
\end{proof}

\begin{rem}\label{rem:SWdef}
The notion of the covering spectrum was first defined by Sormani and Wei in
\cite{SW} for so-called complete length spaces, and the proposition above shows
that their definition exactly gives the set $\Jump(\Fil^\bullet \pi_1(X,x_0)) =
\CovSpec(X) - \{0\}$.  The only difference with our notion of covering
spectrum for complete length spaces is that for non-compact $X$ we sometimes
have $0$ in the covering spectrum such as for the space in Example
\ref{exa:ZeroRadius} (endowed with shortest path length metric).
We note that under our definition, $ 0 \in \CovSpec(X)$ if and only if there
is a covering space $Y \to X$ such that $X$ has open subsets
of arbitrary small diameter that are not evenly covered.
\end{rem}

To conclude this section we will make the filtration of the fundamental group
more explicit for the case of compact Riemannian manifolds. We will identify
generators for the normal subgroups in the filtration: the classes of short
closed geodesics. In later sections, this will allow us to compute the covering
spectra of some familiar classes of manifolds. In particular, in
Example~\ref{Exa:Heis} we will see how to compute the covering spectrum of a
Heisenberg manifold.

Let $M$ be a manifold and let $g$ be a Riemannian metric on~$M$. Choose a base
point $x_0\in M$. In Section~\ref{Sec:Results} we defined the map $m_g\colon\;
\pi_1(M, x_0) \to \nnR$ by sending a class to the infimum of all lengths of
loops that are freely homotopic to it.

\begin{prop}\label{prop:lengthmap}
For every connected manifold $M$ with Riemannian metric
$g$ and base point $x_0\in M$ and every $\delta>0$ we have
$$
\Fil^\delta \pi_1(M,x_0)=
\langle \gamma \in \pi_1(M,x_0)\colon\; m_g(\gamma) <2\delta\rangle.
$$
\end{prop}

\begin{proof}
Suppose that $\si$ is a loop based at $x_0$ in $M$ so that
$m_g([\si])<2\delta$. Then $\si$ is freely homotopic to a loop $\tau$ of length
below~$2\delta$.  It is clear that $\Image(\tau)$ is contained in the open ball
$B_\delta(\tau(0))$ of radius $\delta$ centered at $\tau(0)$.
If $\al$ is the path from $x_0$ to $\tau(0)$, given by the moving base point during
such a homotopy, then $[\si]=[\al * \tau * \overline \al]\in\Fil^\delta
\pi_1(M,x_0)$.
This shows ``$\supset$''.

For the other inclusion, recall that $\Fil^\delta \pi_1(M,x_0)$ is generated
by elements of the form $[\al*\sigma*\overline{\al}]$ where
$\al$ is a path from $x_0$ to a point $x_1\in M$ and $\sigma$ is a loop
based at $x_1$ that is contained in $B_\delta(x_2)$ for some $x_2\in X$.
Since $B_\delta(x_2)$ is locally simply connected we can assume that $\sigma$ has finite
length.  The image of $\sigma$ is compact, so it is in fact contained in
$B_{\delta-\epsilon}(x_2)$ for some $\epsilon>0$.
Choose any path $\al_1$ from $x_1$ to~$x_2$.
We can now break $\sigma$ into finitely many sections of length
less than $\epsilon$, and by connecting the division points with $x_2$
by paths of length smaller than $\delta-\epsilon$ we may write
$[\overline{\al_1}* \sigma* \al_1]\in\pi_1(M,x_2)$ as a finite product
$[\si_1] \cdots [\si_n]$ with
each $\si_i$ a loop based at $x_2$ of length less than
$2(\delta-\epsilon) +\epsilon <2\delta$.
Writing $\beta=\al *{\al_1}$ we now have
\[
[\al*\sigma*\overline{\al}] =
[ \beta *\si_1 *\overline\beta]
[ \beta *\si_2 *\overline\beta]
\cdots
[ \beta *\si_n *\overline\beta]
\]
in $\pi_1(M,x_0)$, and since we have $m_g([\beta *\si_2
*\overline\beta])<2\delta$ for each $i$,
our other inclusion follows.
\end{proof}

\begin{cor}[cf. {\cite[Lemma 4.9]{SW}}]\label{cor:geolength}
Let $(M,g)$ be a compact Riemannian manifold and let
$p\colon\;N\to M$
be a non-trivial covering space. Then the covering radius
$r(N/M)$ is half the length of the shortest closed
geodesic $\sigma$ in $M$ that has a lift to $N$ that is not a closed loop.
\end{cor}
\begin{proof}
By \cite[Lemma 2.4.9]{Spanier} and the fact that the image of $m_g$ is discrete
(see Example~\ref{exa:MinMark}), there is a shortest closed geodesic in $M$
that has a lift to $Y$ that is not a closed loop. Let $\sigma$ be such a
geodesic and let $\delta=m_g([\sigma])$ be its length.
For every $\epsilon>\delta$ the image of $\sigma$ lies within an
open ball of radius $\epsilon/2$, so the covering $p$ is not
$\epsilon/2$-trivial,
and $r(N/M)\le \delta/2$. 
On the other hand, for any loop in $M$ shorter
than $\delta$ all lifts to $N$ are closed loops. Proposition
\ref{prop:lengthmap} then tells us that all loops in $M$ whose homotopy class
in $\Fil^{\delta/2} \pi_1(M)$ have only closed loops as lifts to $N$, and with
Lemma \ref{lem:lpc} this gives $r(N/M)\ge \delta/2$.  
\end{proof}

\begin{proof}[Proof of Proposition \ref{prop:CovJump}]
This is an immediate consequence of Propositions \ref{prop:CovSpecJSet} and
\ref{prop:lengthmap} and the definitions of $\Fil^\bullet$ and $\Fil_{m_g}^\bullet$.
\end{proof}

The moral of Proposition \ref{prop:CovJump} is that it allows
us to compute the covering spectrum of a Riemannian manifold in an easy way as
the jump set of the filtration associated to a length map.  Under the
hypotheses of the proposition, the length map has a closed discrete image in
$\nnR$,
and the group it is defined on, the
fundamental group, is finitely generated.  We now describe how to compute the
jump set in this scenario.

\begin{jsa}[cf. p. 54 of \cite{SW}]\label{jsa}
Let $G$ be a group that is \emph{finitely generated} and let $m: G \to \nnR$ be
a length map with \emph{closed discrete} image.  Then $\Jump(\Fil^\bullet_m G)$ can be
computed as follows. Let
$$
\begin{array}{l}
\delta_1=\inf\{m(g)\colon g \in G, g\ne 1\} \\
\delta_2=\inf\{m(g)\colon g \in G, g\not\in \langle  h \in G\colon\;
m(h)\le\delta_1\rangle \} > \delta_1\\ \delta_3=\inf\{m(g)\colon g \in G,
g\not\in \langle  h \in G\colon\; m(h)\le\delta_2\rangle \} > \delta_2\\
\cdots
\end{array}
$$
where we continue until $\langle  g \in G\colon\; m(g) \le \delta_r\rangle =G$.
The process terminates because $G$ is finitely generated: if $M$ is the largest
value attained by $m$ on some finite set of generators of $G$, then $m$ assumes
only finitely many values up to $M$ (since the image of $m$ is closed and discrete), and the $\delta_i$ are among them.
Then we have
$$
\Jump(\Fil^\bullet_m G)=\{\delta_1,\ldots,\delta_r\}.
$$
In particular we recover the result that the covering spectrum of a compact
manifold is always finite \cite[Theorem 3.4]{SW}.  As an application we show in
the following example how to compute the covering spectrum of a flat torus.
\end{jsa}

\begin{exa}[Flat tori]\label{Exa:FlatTori}
Let $E$ be a Euclidean space, i.e., a vector space over $\R$
of finite dimension $n$ with an inner product~$\inner$.
Let $\mL$ be a lattice in $E$ of full
rank, and let $T$ be the flat torus $T=E/\mL$, where the Riemannian metric is
induced by~$\inner$.  Then we can identify the fundamental group of $T$ with
$\mL$, and the minimum marked length map $m$ on $\mL$ is given by
$m(l)= \|l\|= \sqrt{\langle l, l \rangle}$. We then have
$\CovSpec(T)=\frac{1}{2}\Jump(\Fil_m^\bullet \mL)$, where for $\delta>0$ we
have $\Fil_m^\delta\mL= \langle l: \|l\|< \delta\rangle$.

Those jumps of the filtration where the rank of the sublattice
increases, are the so-called \emph{successive minima} of the
lattice. Counting these jumps with a \emph{multiplicity}, which is the increase
in rank of the sublattice at the jump, we see that there are $n$ successive
minima. However, the covering spectrum of $T$ can have more than $n$ elements:
there can also be jumps where the rank does not increase.

For instance, when $E$ is $5$-dimensional consider a lattice $\mL$ in $E$
spanned by orthogonal vectors $e_1, \ldots , e_5$ where $1 \leq \| e_1 \| <
\cdots < \| e_5 \| < \frac{\sqrt{5}}{2}$.  With the procedure above
we see that the covering spectrum of the flat torus $E / \mL$
is given by $\frac{1}{2}\{ \| e_1 \| , \| e_2, \|, \| e_3\|,
\| e_4 \|, \| e_5 \| \}$.  Now let $v = \frac{1}{2}(e_1 + \cdots + e_5)$ and
consider the lattice $\mL' = \langle \mL, v \rangle$.
The lattice $\mL$ is a sublattice of $\mL'$ of index $2$ and
any vector in $\mL'$ that is not in $\mL$ is of the form $v + w$,
where $w \in \mL$, and its length is at least $\frac{\sqrt{5}}{2}$.
It follows that $\CovSpec(E/\mL')=\CovSpec(E / \mL)\cup\{\frac{1}{2}\|v\|\}$.

We recall from the introduction that the covering spectrum of $M=E/\mL'$ is
also measuring when a particular family $\{\tM^\delta\}_{\delta > 0}$ of
regular covers of $M$ changes isomorphism type.  These coverings, as defined
in \cite[Def. 3.1]{SW}, can be given by $\tM^\delta=E/\Fil^\delta \mL'$.  The
above example highlights the difference between isomorphism type as coverings
and as topological spaces. Indeed, if $ \frac{1}{2}\|e_5\| <\delta \leq
\frac{1}{2}\|v\|$, then $\tM^\delta=E/\mL$, and if $\delta > \frac{1}{2}\|v\|$,
then $\tM^\delta=M$. These are both $5$-tori, but they are distinct regular
covers of $M$.

We revisit the covering spectra of flat tori in Section \ref{Sec:Isospectral}.
\end{exa}


\section{Constructing metrics with a prescribed minimum marked length map}\label{Sec:Prescribing}
The purpose of this section is to give a proof of Theorems 
\ref{thm:IntroSystole} and \ref{thm:varyg}.
As we noted in Example~\ref{exa:MinMark}, every Riemannian metric
$g$ on a closed manifold $M$ gives rise to a length map $m_g\colon\;\pi_1(M)\to
\nnR$, known as the minimum marked length map, sending the homotopy class of a
loop to the infimum of the lengths of the loops that are freely homotopic to
it. The general problem underlying this section is to understand the nature of
those length maps that arise as the minimum marked length map associated to
some Riemannian metric.  Do these length maps satisfy additional properties not
yet implied by the length map axioms of Definition \ref{dfn:lengthmap}?

The next result, whose proof will take up most of this section, implies that
this is not the case if we restrict our attention to properties involving only
finitely many elements.  In addition to helping us establish Theorems
\ref{thm:IntroSystole} and \ref{thm:varyg}, the following result is also
instrumental in the proof of Theorem~\ref{thm:CovSunada} as we will see in
Section~\ref{Sec:Proof}.

\begin{thm}\label{thm:SystoleLengthMap}
Let $M$ be a closed and connected manifold of dimension at least $3$. 
Let $S$ be a finite subset of $\fund$ and suppose that
$m\colon\;S \to \nnR$ is a map satisfying:
\begin{enumerate}
\item[$(i)$] $m(x)>0$ for all non-trivial $x \in S$;
\item[$(ii)$] $m(x) \leq |k|m(y)$ for all $x, y \in S$ and $k\in\Z$ such that $x$ is conjugate
to $y^k$.
\end{enumerate}
Then there is a Riemannian metric $g$ on $M$ such that
$m_g(x)=m(x)$ for all $x\in S$.
In addition, for every $B>0$ the metric
$g$ may be chosen so that $m_g$ satisfies:
\begin{enumerate}
\item{} for all $x\in \fund$ for which
the set $\{|k|m_g(y)\colon\; y \in S,\; k\in \Z \textrm{ and } x \textrm{ is
conjugate to } y^k\}$ is non-empty with infimum $l_x < B$ we have $m_g(x)=l_x$;
\item{} for all other $x\in \fund$ we have $m_g(x)\ge B$.
\end{enumerate}
\end{thm}

\begin{rem}\label{rem:SystoleLengthMap}
By Definition \ref{dfn:lengthmap} conditions $(i)$ and $(ii)$ are clearly
necessary for the existence of an extension of $m$ to a length map on $\fund$.
The theorem above implies that they are also sufficient, and that this
extension can be taken to be $m_g$ for a suitable metric $g$.
Conditions $(1)$ and $(2)$ describe the degree of control we have over 
the behavior of $m_g$ outside $S$.
In particular, this control allows us to ensure that in $(M,g)$ the
unoriented free homotopy classes (see Definition~\ref{def:MinLSpec}) containing
geodesics of length less than $B$ can only be found \emph{among} those
classes corresponding to elements of the form $y^k$ for some $y \in S$ and
$k \in \Z$: condition $(1)$ 
determines the collection of such ``short'' classes precisely. 
We will use this
in the proof of Theorem~\ref{thm:varyg}. 
\end{rem}

Let $M$ be a manifold and let $\mM$ denote the space of all Riemmannian
metrics on~$M$. Given two metrics $g, h \in \mM$ we will say that $g$ is
\emph{larger than} $h$, denoted $g \succeq h$, if for any tangent vector 
$v \in TM$ we have
$g(v,v) \geq h(v,v)$. Clearly, $\preceq$ defines a partial order on $\mM$.
The set $\mM$ is closed under addition, and under multiplication
by everywhere positive smooth functions on~$M$.

\begin{proof}[Proof of Theorem~\ref{thm:SystoleLengthMap}]
First, we note that by increasing $B$ we can reduce to the case that
$B > m(s)$ for all $s\in S$.

Let the manifold $S^1=\R/\Z$ be the standard circle. Since the dimension of $M$
is at least $3$, we can represent the free homotopy classes of the elements
$s\in
S$ by smooth embeddings
$\si_s\colon\; S^1\to M$ with pairwise disjoint images.
The tubular neighborhood theorem \cite[Thm.~5.2, Ch.~4]{Hirsch} says that for
each $s\in S$ there is a smooth vector bundle $N_s$ over $S^1$ together with a
diffeomorphism $i_s\colon\;N_s\to T_s \subset M$ onto an open subset $T_s$ of
$M$, whose composition with the zero
section $S^1\to N_s$ is the map~$\si_s$. The loops $\si_s$ have positive
distance between each other, so by shrinking the neighborhoods, if necessary, we
can also assume that the tubular neighborhoods $T_s=\im(i_s)$ are disjoint.

Let us consider the structure of such a tubular neighborhood $T_s$.  If $n$
is the dimension of $M$, then the bundle $N_s$ on $S^1$ has rank $n-1$.
There are up to isomorphism exactly two vector bundles of rank $r$ over $S^1$:
one orientable and one non-orientable \cite[Ch.~4, Sec.~4, Ex.~2]{Hirsch},
\cite[Chp. 5]{Ranicki}. In view of choosing metrics later, let $B^{n-1}$ be
the standard open unit ball around $0$ in $\R^{n-1}$, which is diffeomorphic to
each fiber of the vector bundle $N_s$ over $S^1$.
Now consider the quotient $(B^{n-1}\times \R)/\Z$, where 
$n\in \Z$ acts by sending $((x_1,x_2,\ldots, x_{n-1}) ,t)$
to $((x_1,x_2,\ldots, x_{n-1} ),t+n)$ if $N_s$ is orientable,
and to $(((-1)^n x_1,x_2,\ldots, x_{n-1}) ,t+n)$ if $N_s$ is non-orientable.
Then it follows that there is a commutative diagram:

\begin{center}
\begin{tikzpicture}
\matrix [column sep=14mm,row sep=14mm,matrix of math nodes,ampersand replacement=\&]
{
                   \& |(T)|  S^1           \& \\
  |(B1)| (B^{n-1}\times\R)/\Z \& |(B2)| N_s \& |(B3)| T_s \\
};
\draw
      (T)     edge [->,bend right] node [above left]{$t\mapsto (0,t)$}(B1)
              edge [->,bend left] node [above right]{$\si_s$}(B3)
              edge [->] node [left]{$0$-section}(B2)
      (B1)    edge [->] node [above] {$\sim$}      (B2)
      (B2)    edge [->] node [above] {$\sim$}      (B3)
              edge [->] node [below] {$i_s$}      (B3);
\end{tikzpicture}
\end{center}

\noindent
where the horizontal maps are diffeomorphisms. 

In both the orientable and non-orientable case, the diffeomorphisms give rise
to the following additional structure on~$T_s$.  First, the standard Riemannian
metric on $(B^{n-1}\times \R)/\Z$ gives a Riemannian metric $h_0$ on $T_s$ that
has the property that for every $k\in \Z$ each loop within $T_s$ homotopic to
$\si_s^k$ has length at least $|k|$, with equality for the loop $\si_s^k$
itself.  Second, we obtain a smooth map $r_s\colon\; T_s\to [0,1)$ by sending the
point associated to $(x_1,\ldots, x_{n-1} ,t)$ to the length of the vector
$(x_1,\ldots, x_{n-1})\in B^{n-1}$. Note that $r_s$ is $0$ on $\Image(\si_s)$.

For each $s\in S$ we will consider five open neighborhoods
$T^i_s=\{x\in T_s\colon\; r_s(x)<i/5\}$ on $\Image(\si_s)$ for $i=1,\ldots,5$,
where $T_s^5=T_s$. We write $T=\bigcup_s T_s$ and $T^i=\bigcup_sT_s^i$,
and we let $r\colon T \to [0,1)$ be the smooth map given by $r_s$ on each~$T_s$.

Let $h_1$ be the metric on $T$ whose restriction to the tubular neighborhood
$T_s$ is given by $m(s)\cdot h_0$.  Then with respect to the metric $h_1$, the
loop $\sigma_s$ has length $m(s)$, for each $s \in S$.
Now, let $\kappa \geq 1$ be a constant such that with respect to
$\kappa \cdot h_1$ the distance between $T-T^3$ and $T_2$ is at least~$B$.
Since $M$ is compact, we may fix a Riemannian metric $g_0$ on $M$ such that
every non-contractible loop in $M$ has length at least~$B$; that is, we may
choose $g_0$ so that its systole is at least $B$.

\begin{lem}
With the notation and assumptions as above, 
there is a Riemannian metric $g$ on $M$ that satisfies the following
properties:
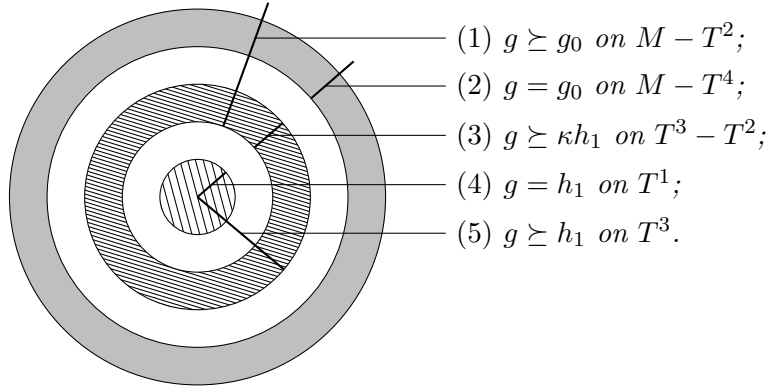
\begin{center}
\begin{tikzpicture}[scale=0.5]

\draw [fill=lightgray] (0,0) circle (5cm);
\draw [fill=white] (0,0) circle (4cm);
\begin{scope}\clip (0,0) circle (3cm);
  \draw[xstep=0.1cm,ystep=10cm, rotate=73, xshift=-5cm, yshift=-5cm] (0,0) grid
(10,10); \end{scope}
  \draw (0,0) circle (3cm);
\draw [fill=white] (0,0) circle (2cm);
\begin{scope} \clip (0,0) circle (1cm);
  \draw[xstep=0.2cm,ystep=10cm, rotate=15, xshift=-5cm, yshift=-5cm] (0,0) grid
(10,10); \end{scope}
  \draw (0,0) circle (1cm);

\draw[thick] (70:2cm) -- (70:5.5cm);
\node[anchor=west] (a) at (70:4.5cm -| 6.6,0) {$(1)$ $g \succeq g_0$ on
$M-T^2$;}; \draw (70:4.5cm) -- (a);

\draw[thick] (41:4cm) -- (41:5.5cm);
\node[anchor=west] (a) at (41:4.5cm -| 6.6,0) {
$(2)$ $g = g_0$ on $M-T^4$;
};
\draw (41:4.5cm) -- (a);

\draw[thick] (41:2cm) -- (41:3cm);
\node[anchor=west] (b) at (41:2.5cm -| 6.6,0) {
$(3)$ $g \succeq \kappa h_1$ on $T^3- T^2$;
};
\draw (41:2.5cm) -- (b);

\draw[thick] (41:0cm) -- (41:1cm);
\node[anchor=west] (d) at (41:0.5cm -| 6.6,0 ) {
$(4)$ $g = h_1$ on $T^1$;
};
\draw (41:0.5cm) -- (d);

\draw[thick] (-40:3cm) -- (0,0);
\node[anchor=west] (c) at (-40:1.5cm -| 6.6,0) {
$(5)$ $g \succeq h_1$ on~$T^3$.
};
\draw (-40:1.5cm) -- (c);

\end{tikzpicture}
\end{center}

\end{lem}

\begin{proof}[Proof of Lemma]
First consider the metric $h_2= g_0 + \kappa h_1$ on~$T$: it clearly satisfies
$h_2 \succeq g_0$ on $T$.  Now, let $f_1\colon\; [0,1] \to [0,1]$ be a smooth
function
with $f_1(t)=1$ for $t \leq 3/5$ and $f_1(t)=0$ for $t \geq 4/5$.
We define the metric $\hat g$ by gluing two metrics on open subsets of
$M$ that coincide on their intersection $T-T^4$:
\[
\hat g= \left\{ \begin{array}{ll}
            (f_1 \circ r)h_2 + (1- (f_1\circ r)) g_0 &\text{ on } T \\
            g_0 &\text{ on } M -T^4 \\
            \end{array}
            \right.
\]
On $M$ we have $\hat{g} \succeq g_0$, and on $T^3$ we have $\hat{g} = h_2
\succeq \kappa \cdot h_1 \succeq h_1$.

Now let $f_2\colon\; [0,1] \to [0,1]$ be a smooth function such that
$f_2(t) =1$ for $t \leq 1/5$ and $f_2(t) = 0$ for $t \geq 2/5$ and set
$$ g = \left\{ \begin{array}{ll}
            (f_2 \circ r)h_1 + (1- (f_2\circ r)) \hat{g} & \mbox{on $T$} \\
            \hat g & \mbox{on $M - T^2$}
            \end{array}
            \right.
$$
Then $g$ satisfies properties (1)--(5).
\end{proof}

\noindent
We claim that the metric $g$ produced above has the desired properties.

Indeed, let $x \neq 1 \in \pi_1(M)$, and suppose $c$ is a loop of shortest length in the
free homotopy class associated to $x$: $c$ will necessarily have length~$m_g(x)$.
Then we are in at least one of the following three cases.

{\bf Case A.} If $\Image(c) \subset T^3$,
then there is a unique $s\in S$ such that $\Image(c) \subseteq T_s^3$.
Then $c$ is freely homotopic within $T_s$ to $\sigma_s^k$ for some $k \neq 0
\in \Z$.  We have $g \succeq h_1$ on $T_s^3$, so $c$ has length
at least $|k|m(s)$ with respect to~$g$.

{\bf Case B.} If $\Image(c) \subset M - T^2$,
then we use that
$g \succeq g_0$ on $M - T^2$ to see that $c$ has length at least~$B$.

{\bf Case C.} If $\Image(c)$ contains points in $T^2$ and in $M- T^3$,
then the length of $c$ is at least the distance measured with metric $g$
between $T^2$ and $M- T^3$. But since $g\succeq \kappa h_1$ on $T^3-T^2$
the definition of $\kappa$ shows that this distance is at least~$B$.

To finish the proof, we first notice that $\si_s^k$ has length $|k|m(s)$ by
property (4) of the lemma.
Now, suppose $x \in \fund$, with associated loop $c$ (as above), is such that
the set  $\mathcal{S}_x =\{ |k|m(s) : s\in S, k\in \Z \text { with }\si_s^k
\text{
freely homotopic to }c \}$ is nonempty with infimum~$l_x$.
If this infimum is below $B$, then we must be in case A, and $l_x=m_g(x)$.
In particular, for $x\in S$ we know that $l_x$ is less than $B$ by
our assumption at the start of the proof, and $l_x=m(x)$ by the
assumption on $m$ in the Theorem, and it follows that $m(x)=m_g(x)$.
If the set $ \mathcal{S}_x$ is non-empty and its infimum is at least $B$, 
then $m_g(x)\ge B$ in all three cases.
If $ \mathcal{S}_x$ is empty, then we cannot be in case $A$ and we
also deduce that $m_g(x)\ge B$. This completes the proof of
Theorem~\ref{thm:SystoleLengthMap}.
\end{proof}

\begin{proof}[Proof of Theorem~\ref{thm:IntroSystole}] 
In the first part of Remark~\ref{rem:SystoleLengthMap} we established that
conditions $(i)$ and $(ii)$ in the hypotheses of
Theorem~\ref{thm:SystoleLengthMap} are equivalent to $m$ being the restriction
of a length map on $\fund$. Hence, we see that Theorem~\ref{thm:IntroSystole}
is just the first statement of Theorem~\ref{thm:SystoleLengthMap}.
\end{proof}

Before proving Theorem~\ref{thm:varyg} we recall that its moral is that we can
find a metric $g$ with a prescribed value and location for the first $k$
elements of the minimum marked length spectrum (see
Definition~\ref{def:MinLSpec}). Now, given distinct unoriented free homotopy
classes $c_1, c_2, \ldots , c_k$ and a minimum marked length map $m_g$, the
length map axioms imply certain relations between the values $l_i =
m_g(c_i)$.  Examining conditions $(ii)$ and $(1)$ in the statement of
Theorem~\ref{thm:SystoleLengthMap} and keeping the moral of
Theorem~\ref{thm:varyg} in mind we are led to make the following definition.  

\begin{dfn}\label{dfn:Admissible}
Let $M$ be a closed manifold and let $\mF(M)$ denote its collection of
unoriented free homotopy classes.  Given a finite sequence $\Fvar= (c_1, c_2,
\ldots , c_k)$ of distinct elements of $\mF(M)$ with $c_1$ the trivial class,
we will say that a sequence $l_1= 0 < l_2 \leq l_3 \leq \cdots \leq l_k$ is
\emph{$\Fvar$-admissible} if it satisfies the following conditions:
\begin{enumerate}
\item for $i, j = 2, \ldots, k$ we have $l_i \leq |n| l_j$ whenever $c_i =
c_j^n$ for some $n \in \Z$; 
\item for $i = 2, \ldots, k$ we have $l_i \geq \frac{1}{|n|}l_k$ whenever
$c_i^n \not\in \{\fvar_1,\ldots,\fvar_k\}$ for some non-zero $n\in \Z$.
\end{enumerate}
\end{dfn}

\begin{exa}[$\Fvar$-admissibility]\label{exa:Admissible}
One can quickly verify that a sequence $l_1= 0 < l_2 \leq l_3 \leq \cdots \leq
l_k$ such that $2 l_2 \geq l_k$ is $\Fvar$-admissible for any choice of
$\Fvar= (c_1, c_2, \ldots , c_k)$ in the definition above.
Also, we note that if $\{\fvar_1,\ldots,\fvar_n\}$ is a subset
that is closed under taking powers, 
the definition of $\Fvar$-admissibility reduces to condition $(1)$
(cf. Remark~\ref{rem:SystoleLengthMap}).
\end{exa}

\begin{proof}[Proof of Theorem~\ref{thm:varyg}]
To prove that $(5)$ implies $(6)$ 
first
recall from homotopy theory that 
two homotopy classes $x,y\in \fund$ give rise to the same
unoriented free homotopy class in $\mF(M)$
if and only if $x$ is conjugate to $y$ or to $y^{-1}$ in $\fund$. 
As a consequence, any minimum marked length map $m_g: \fund \to \nnR$ 
induces a map $m_g : \mF(M) \to \nnR$ called the minimum marked length
spectrum (cf. Definition~\ref{def:MinLSpec}).

Now, as in the statement of Theorem~\ref{thm:varyg}, let 
$\Fvar = (\fvar_1, \ldots , \fvar_k )$ be a sequence of
$k$ distinct unoriented free homotopy classes, where $c_1$ is trivial
and let $0=\lvar_1 < \lvar_2 \leq \cdots \leq \lvar_k$ be a
$\Fvar$-admissible sequence. 
For each $i = 1, \ldots , k$ choose $x_i \in \fund$ within the class $\fvar_i$.
Let $S = \{x_1, \ldots , x_k \} \subset \fund$ and
define $m\colon\;S \to \nnR$ by $m(x_j) = \lvar_j$.
Then it follows from condition $(1)$ of $\Fvar$-admissibility that $m$ satisfies 
the hypotheses of Theorem~\ref{thm:SystoleLengthMap}. 
Hence, choosing $B \geq \lvar_k = \max_{j} m(x_j)$ in Theorem~\ref{thm:SystoleLengthMap} 
it follows from condition $(2)$ of
$\Fvar$-admissibility that we obtain a metric $g$ with the desired properties.

To prove that $(6)$ implies $(5)$, 
let $g$ be a metric on $M$ such that
$l_1 = m_g(c_1) = 0 < l_2 = m_g(c_2) \leq l_3 = m_g(c_3) \leq \cdots \leq l_k =
m_g(c_k)$ and $m_g(c) \geq l_k$ for all $c \in \mF(M) -
\{\fvar_1,\ldots,\fvar_k\}$. Then an examination of the length map axioms
allows us to conclude that the sequence $l_1 = 0 < l_2 \leq \cdots \leq l_k$ is
$\Fvar$-admissible. 
\end{proof}

\begin{rem}
The statement and proofs of Theorems~\ref{thm:SystoleLengthMap} and
\ref{thm:varyg}  
also hold when $M$ is a closed manifold of dimension $2$ and the sequence of free
homotopy classes associated to the
elements of $S$ can be represented by disjoint simple closed loops in~$M$: a
condition that is always met in dimension $3$ and higher. It is worth noting
that Colin de Verdi\`{e}re's result \cite{CdV2} on prescribing the Laplace spectrum 
fails in dimension $2$ precisely because not every finite collection of free homotopy classes can be
represented by a collection of pairwise disjoint simple closed curves \cite[p.
415]{Berger}.
\end{rem}

We close this section by noting that in the case where $\pi_1(M)$ is finite, we
may take $S$ to be all of $\pi_1(M)$ in Theorem~\ref{thm:SystoleLengthMap} and
$m\colon\; S \to \nnR$ to be any length map in the sense of
Definition~\ref{dfn:lengthmap}. Then since every finitely generated group is
the fundamental group of some closed $4$-manifold \cite[Exercise
4.6.4(b)]{GoSt} we obtain the following result, which says that every length
map on a finite group is the minimum marked length map of some Riemannian
manifold.

\begin{cor}\label{cor:SystoleLengthMap}
Let $m\colon\; G \to \nnR$ be a length map on a finite group $G$.
Then there exists a Riemannian manifold $(M,g)$ and an isomorphism
$\phi\colon\;G\to \fund$ so that $m = m_g\circ \phi$.
\end{cor}


\section{A group theoretic criterion for equality of covering
spectra}\label{Sec:Proof}

The aim of this section is to establish Theorem~\ref{thm:CovSunada}.
First, we identify how the covering spectrum of the total space
of a Riemannian covering can be computed in terms of the base space.

Let $p\colon\;(M,g)\to (N,h)$ be a Riemannian 
covering; that is, $p$ is a covering map that is also a local isometry.
Fixing a base point $m_0$ of $M$ and $p(m_0)$ of $N$, we
have an injective group homomorphism $p_{\#} \colon\;\pi_1(M) \to \pi_1(N)$
given by $[\sigma] \mapsto [p \circ \sigma]$.
Since $p$ is a local isometry, the lengths of $\si$ and $p\circ\sigma$ are the same, so we have
the following commutative diagram:

\begin{center}
\begin{tikzpicture}
\node (N) at (-1.7,2) {$\pi_1(M)$};
\node (NN) at (1.7,2) {$\pi_1(N)$};
\node (R) at (0,0.3) {$\nnR$};
\draw [->] (N) -- node[above] {$p_\#$} (NN);
\draw [->] (NN) -- node[below right] {$m_{h}$}(R);
\draw [->] (N) -- node[below left] {$m_{g}$}(R);
\end{tikzpicture}
\end{center}

\noindent
From this the following observation is immediate.

\begin{lem}
\label{lem:restrict}
$\CovSpec(M)=\frac{1}{2}
\Jump(\Fil^\bullet_{m_{h}}\Image(p_\#))$.
\end{lem}

\noindent
Hence, if $p: (M , g) \to (N, h)$ and $p: (M' , g') \to (N, h)$ are Riemannian
coverings of a common base space $(N,h)$, then $\CovSpec(M, g) = \CovSpec(M',
g')$ if and only if $\Jump(\Fil^\bullet_{m_{h}}H) = \Jump(\Fil^\bullet_{m_{h}}
H')$, where $H = \Image(p_\#)$ and $H' = \Image(p'_\#)$ are subgroups of $G =
\pi_1(N)$.

\begin{rem}
A word of warning is in order here. Suppose that $m$ is a length map on a
group $G$ and $H$ is a subgroup of $G$, then the restriction of
$m$ to $H$ is a length map on $H$ that gives rise to the filtration
$\Fil^\bullet_mH$. Note however that
$\Fil_{m}^\delta H $ may not be the same as $(\Fil_m^\delta G)\cap H$.
The first group consists of finite products of elements of length below
$\delta$ where all elements are in $H$, and the second group consists of such
products where the elements are in $G$ and only the product is required to be
in~$H$. Thus, to obtain the correct covering spectrum of $M$ out of the length
map $m_h$ on $N$ one has to first restrict $m_h$ to $\Image(p_\#)$, rather than
restricting the filtration of $\pi_1(N)$ to $\Image(p_\#)$.
\end{rem}

In order to prepare for the proof of Theorem \ref{thm:CovSunada}, we
first give an interpretation of condition $(3)$
of the theorem in terms of length maps.

\begin{lem}\label{lemma:jumptriple}
Let $G$ be a finite group, and let $H$ and $H'$ be subgroups.
Then the following are equivalent:
\begin{enumerate}
\item
for all subsets $S \subset  T \subset G$ that are
stable under conjugation by elements of $G$ we have
\[
\langle H\cap S\rangle = \langle H\cap T\rangle
\iff
\langle H'\cap S\rangle = \langle H'\cap T\rangle;
\]
\item{}
for every length map $m$ on $G$ we have
$$\Jump(\Fil^\bullet_{m}H)=
\Jump(\Fil^\bullet_{m}H').$$
\end{enumerate}
\end{lem}

\begin{rem}\label{rem:range}
There is some freedom in choosing the range of conjugacy stable subsets of $G$
that $S$ and $T$ vary over in condition~$(1)$. For instance, one obtains an
equivalent condition when one lets $S$ and $T$ vary over all conjugacy stable
subsets of~$G$, so that condition $(1)$ in the lemma is equivalent to condition
$(3)$ of Theorem \ref{thm:CovSunada}.
\end{rem}

\begin{dfn}\label{dfn:jump}
We will say that $H$ and $H'$ are \emph{jump equivalent} subgroups of $G$ or
that $(G,H,H')$ is a \emph{jump triple}, if the conditions in the lemma above
are met. 
\end{dfn}

\noindent
In the next section we will see how the notion of jump equivalence
is related to other equivalence relations of subgroups.

\begin{proof}[Proof of Lemma]
Suppose that $(1)$ holds and that $m$ is a length map on~$G$.
For every $\delta>0$ the subset $S_{\delta}=\{g\in G: m(g)<\delta\}$ is stable
under conjugation.
We have $\{h \in H\colon\; m(h) < \delta\}= H\cap S_\delta$, so
$$
\delta\in \Jump(\Fil^\bullet_{m}H) \iff
\langle H\cap S_\delta \rangle \ne
\langle H\cap S_{\eps} \rangle \textrm{ for all }\eps > \delta.
$$
By condition $(1)$ we see that this is equivalent to the same statement for $H'$,
and $(2)$ follows.

To show that  $(2)$ implies $(1)$ suppose that $S, T \subset G$ are conjugacy
stable subsets of $G$, and that $S\subset T$.  We will construct a length map
$m\colon\; G \to \R_{\ge 0}$ that depends on $S$ and $T$ so that condition
$(2)$ for this length map implies that the equivalence in $(1)$ holds.

We may assume that $S$ and $T$ are closed under taking inverses.
Define $m$ by
$$
m(g) = \left\{ \begin{array}{ll}
            0 & \mbox{if $g =1$;} \\
            2 & \mbox{if $g \in S - \{ 1 \}$;} \\
            3 & \mbox{if $g \in T - (S \cup \{1\})$;}\\
            4 & \mbox{otherwise.}
\end{array}
\right.
$$
Then $m$ satisfies the length map axioms of Definition \ref{dfn:lengthmap}, and
we have $(H \cap S) \cup \{1\} =\{h\in H\colon\; m(h) <3\}$,
and $(H \cap T) \cup \{1\} =\{h\in H\colon\; m(h) <4\}$.
It follows that
$$
\langle H \cap S \rangle = \langle H \cap T \rangle \iff
\Fil^3_{m}H= \Fil^4_{m}H \iff
3 \not\in \Jump(\Fil^\bullet_{m}H).
$$
The same holds when we replace $H$ by $H'$, so indeed condition $(2)$ for $m$
implies the equivalence in $(1)$.
\end{proof}

\noindent
With these preliminaries out of the way we now turn to the proof of
Theorem~\ref{thm:CovSunada}.

\begin{proof}[Proof of Theorem \ref{thm:CovSunada}]
First, we notice that since $G$ acts freely on $M$ we see that the natural
projection $q\colon\; M \to G\backslash M$ is a regular covering with $G$ serving as
the group of deck transformations. 
Let us choose a base point in $M$ for the fundamental group $\fund$, which
also determines a base point in the quotient of $M$ by any subgroup of $G$.
By \cite[Cor. 2.6.3]{Spanier} there is a surjective
homomorphism $\phi\colon\; \pi_1(G \backslash M) \to G$ with $\ker(\phi) = q_{\#}
(\pi_1(M))$.
Hence, $\phi$ is an isomorphism if and
only if $M$ is simply-connected. We also note that, letting $p\colon\; H \backslash M
\to G\backslash M$ and $p'\colon\; H' \backslash M \to G \backslash M$ denote the
covering maps, we have
$p_{\#}(\pi_1(H \backslash M)) = \phi^{-1}(H)$ and
$p'_{\#}(\pi_1(H' \backslash M)) = \phi^{-1}(H')$.

We now prove that if $M$ is simply-connected, then $(3)$ implies $(4)$. Assume
that the triple $(G, H, H')$ satisfies condition $(3)$. Then, since $\phi \colon\;
\pi_1(G\backslash M) \to G$ is an isomorphism, we see that condition $(3)$
holds for the triple $(\pi_1(G \backslash M), p_{\#}(\pi_1(H \backslash M)),
p'_{\#}(\pi_1(H' \backslash M)))$. The result now follows by applying Lemma
\ref{lemma:jumptriple} and Lemma \ref{lem:restrict} to the coverings $p\colon\; M\to
H\backslash M$ and $p'\colon\;M\to H'\backslash M$.

We now assume that $M$ is a manifold of dimension at least $3$ and prove that
in this case condition $(4)$ implies condition $(3)$.

Assuming condition $(3)$ does not hold, our goal will be
to construct a Riemannian metric on $G\backslash M$ such that the
length map attains certain prescribed values on a suitably chosen finite set of
conjugacy classes of $\pi_1(G\backslash M)$. Pulling this metric back to $M$
will then give a $G$-invariant
Riemannian metric on $M$ so that the quotient manifolds $H\backslash M$ and
$H'\backslash M$ have distinct covering spectra.

Note first that $\pi_1(G\backslash M)$ is finitely generated, and that
$G$ is finite, so $\ker \phi$ is generated as a group by a finite
subset~$K$ of $\pi_1(G\backslash M)-\{1\}$; see e.g. \cite[Cor.~7.2.1]{Hall}.
By the assumption that condition $(3)$ does not hold, one sees that there are
subsets $S$ and $T$ of $G-\{1\}$ both stable under conjugation by elements of
$G$, and taking inverses, so that we have $S\subset T$, and $\langle H\cap
S\rangle =\langle H \cap T\rangle$,  but $\langle H'\cap S\rangle  \ne \langle
H' \cap T\rangle$ (switch $H$ and $H'$ if necessary).
Now choose any lift $T'$ of $T$ to $\pi_1(G\backslash M)$, i.e., any subset
$T'\subset \pi_1(G\backslash M)$ so that $\phi$ maps $T'$ bijectively to~$T$.
Let $S'=\{t\in T'\colon\; \phi(t)\in S\}$, which is a lift of~$S$.

We will prescribe the length map on the finite subset $K\cup T'$ of
$\pi_1(G\backslash M)$. Note that this union is a disjoint union.
For any subset $X$ of $\pi_1(G\backslash M)$ let us write $c(X)=\{\gamma
y\gamma ^{-1}\colon\; \gamma \in \pi_1(G\backslash M),\; y \in X\}$ for its
closure under conjugation.
Since we assumed that the dimension of $M$ is at least $3$,
Theorem~\ref{thm:SystoleLengthMap} now implies that there is a Riemannian
metric $g$ on $G\backslash M$ such that
the associated minimum marked length map $m_g\colon\; \pi_1(G\backslash M)\to
\nnR$ satisfies \[
m_g(\gamma)\;
\left\{
\begin{array}{ll}
= 0 &\text{ if  $\gamma=1$;} \\
= 3 & \text{ if $\gamma\in c(K\cup S')$;} \\
= 4 & \text{ if $\gamma\in c(T'-S')$.} \\
> 5 & \text{ otherwise}
\end{array}
\right.
\]
For any conjugacy stable $X \subset \pi_1(G\backslash M)$ that contains $K$
the group $\langle X\cap \phi^{-1}(H)\rangle$ contains
$\ker \phi=\langle K \rangle$, and its image under $\phi$ is
$\langle \phi(X) \cap H\rangle$, so we have
\[
\langle X\cap \phi^{-1}(H)\rangle = \phi^{-1}(\langle \phi(X) \cap H \rangle).
\]
Since $\langle H \cap S \rangle =\langle H \cap T\rangle$,
we see that for every $\delta$ with $4<\delta<5$ we have
\begin{align*}
\Fil^4_{m_g}\phi^{-1}(H)
&= \langle c(K\cup S') \cap \phi^{-1}(H)\rangle = \phi^{-1}(\langle S\cap H
\rangle) \\ &= \phi^{-1}(\langle T\cap H\rangle) = \langle c(K\cup T') \cap
\phi^{-1}(H)\rangle
=\Fil^\delta_{m_g}\phi^{-1}(H),
\end{align*}
and we conclude that $4$ is not in $\Jump(\Fil_{m_g}^{\bullet} \phi^{-1}(H))$.
On the other hand, since $\langle H'\cap S\rangle  \ne \langle H'\cap
T\rangle$, we have for each $4 < \delta <5$ \[
\Fil^4_{m_g}\phi^{-1}(H') = \phi^{-1}(\langle S\cap H'\rangle) \ne
\phi^{-1}(\langle T\cap H'\rangle ) = \Fil^\delta_{m_g}\phi^{-1}(H').
\]
Therefore, $4$ is an element of $\Jump(\Fil_{m_g}^{\bullet} \phi^{-1}(H'))$.
Now, since we have $p_{\#}(\pi_1(H \backslash M)) = \phi^{-1}(H)$ and
$p'_{\#}(\pi_1(H' \backslash M)) = \phi^{-1}(H')$, it follows from
Lemma \ref{lem:restrict} that $2$ lies in the covering spectrum 
of $H'\backslash M$ but not in the covering spectrum of $H\backslash M$.
This completes the proof of Theorem \ref{thm:CovSunada}.
\end{proof}


\section{Jump equivalence and Gassmann-Sunada equivalence}\label{Sec:Groups}

Suppose that $G$ is a finite group, and let $H$ and $H'$ be subgroups.  In this
entirely group theoretic section we study the notion of jump equivalence of
$H$ and $H'$, as defined in Definition \ref{dfn:jump}.

Like Gassmann-Sunada equivalence, jump equivalence is a strong form of
the group theoretic notion corresponding to Kronecker equivalence of
number fields \cite{Jehne}.
Many instances of Gassmann-Sunada equivalence, such as the Komatsu triples
considered in \cite{SW}, satisfy a stronger condition, which we call
\emph{order equivalence}. The definitions and relations of these four
equivalences are given in Figure~\ref{fig:SubgroupEquiv}, where the stated conditions are to
hold for all subsets $S$, $T$ of $G$ that are closed under conjugation by
elements of~$G$.
It is not hard to show the implications in the diagram---we leave this
to the reader. The main purpose of this section is to provide examples
of Gassmann-Sunada equivalent subgroups that are not jump equivalent,
which in the next section give rise to isospectral manifolds with distinct
covering spectra.  In fact, we will see that there are no relations between the
four equivalences that hold for all $(G,H,H')$ other than the ones implied by
the diagram. We start with a few basic remarks.

\bigskip
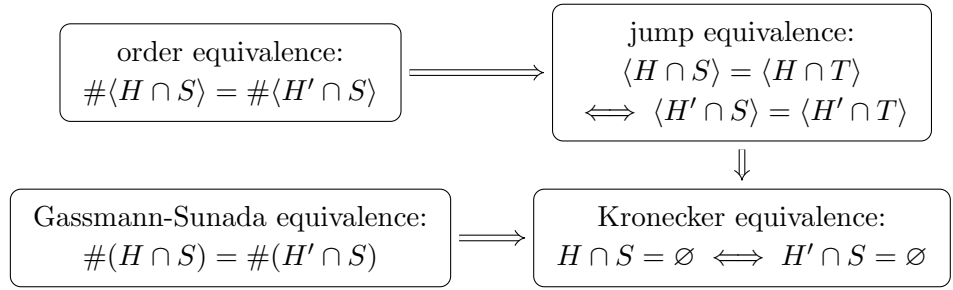
\begin{figure}
\begin{center}
\begin{tikzpicture}
\path
(-7,2.2) node[rounded corners, draw, rectangle] (OE) {
\begin{tabular}{c}
order equivalence:\\
$\#\langle H \cap S\rangle = \#\langle H' \cap S\rangle$
\end{tabular}
}
(0,2.2) node[rounded corners, draw, rectangle] (JE) {
\begin{tabular}{c}
jump equivalence:\\
$\langle H \cap S\rangle =
\langle H \cap T\rangle  \\ \iff
\langle H' \cap S\rangle =
\langle H' \cap T\rangle $\\
\end{tabular}
}
(0,0) node[rounded corners, draw, rectangle] (KE) {
\begin{tabular}{c}
Kronecker equivalence:\\
$H \cap S =\emptyset \iff H' \cap S =\emptyset $\\
\end{tabular}
}
(-4,0) node[rounded corners, draw, rectangle, anchor=east] (GE) {
\begin{tabular}{c}
Gassmann-Sunada equivalence:\\
$\# (H \cap S) =
\# (H' \cap S)$\\
\end{tabular}
};
\draw[-implies,double,double equal sign distance,shorten >= 3pt,shorten <= 3pt]
(OE)--(JE);
\draw[-implies,double,double equal sign distance,shorten >= 3pt,shorten <= 3pt]
(JE)--(KE);
\draw[-implies,double,double equal sign distance,shorten >= 3pt,shorten <= 3pt]
(GE)--(KE);
\end{tikzpicture}
\end{center}
\caption{Notions of equivalence}\label{fig:SubgroupEquiv}
\end{figure}

\begin{rem}
One obtains an equivalent condition for jump equivalence by restricting the
condition to the case where $S\subset T$, or even to the case where $T$ is the
union of $S$ and a single conjugacy class of~$G$.  
For Gassmann-Sunada equivalence and Kronecker equivalence, but not for order
equivalence, one may restrict to subsets $S$ consisting of a single conjugacy
class.
\end{rem}

\begin{rem}[Preserving the index]\label{rem:PresIndex}
If $H$ and $H'$ are order equivalent or Gassmann-Sunada equivalent, then
$[G:H]=[G:H']$. Thus, condition (1) in Theorem \ref{thm:Sunada}
implies that $H\backslash M$ and $H'\backslash M$ are coverings of the same
degree of $G\backslash M$. For Kronecker and jump equivalence
this is not necessarily so.  The easiest example is the jump triple
$(A_4,V_4,C_2)$: the alternating group $A_4$ on four letters,
its unique subgroup $V_4$ of order $4$ and any subgroup $C_2$ of order~$2$.
\end{rem}

\begin{rem}[Reduced triples]\label{rem:ReducedTriples}
Another property that holds for Gassmann-Sunada equivalence, but not
for jump equivalence is the following. Suppose that $N$ is a normal subgroup of
$G$ that is contained in $H\cap H'$. If the only such $N$ is the trivial group,
then we say that the triple $(G,H,H')$ is \emph{reduced}.
It is easy to see that $(G,H,H')$ is a Gassmann-Sunada
triple if and only if $(G/N,H/N,H'/N)$ is a reduced Gassmann-Sunada triple,
where $N \leq H \cap H'$ is a maximal normal subgroup.
The same is true for Kronecker equivalence.
For jump equivalence the situation is different. If $(G,H,H')$ is a jump
triple, then so is $(G/N,H/N,H'/N)$, but the converse may not hold.
For instance, 
multiplying all groups in the example in the previous remark by $C_2$,
the reader may check that $(C_2\times A_4, C_2\times V_4, C_2\times C_2)$ is not
a jump triple, while dividing out $C_2\times\{1\}$ gives the jump triple
$(A_4,V_4,C_2)$. We will see more examples of this below.
\end{rem}

\begin{exa}[Small index]\label{exa:small}
For many Gassmann-Sunada triples $(G,H,H')$ there is an automorphism $\al$
of $G$ such that $H'=\alpha(H)$ and $\alpha(g)$ is conjugate to $g$ for each $g
\in G$. For, instance this is the case for all 19 reduced Gassmann-Sunada
triples $(G,H,H')$ with $[G:H]< 16$; see \cite[Thm.~3]{Bosma}.
Such subgroups, $H$ and $H'$, are order equivalent and therefore also jump equivalent in $G$.
\end{exa}

\begin{exa}[Linear groups]
One way to obtain Gassmann-Sunada triples is the following.
Let $k$ be a finite field of $q$ elements and let $V$ be a vector space of
dimension~$d$.  Let $G=\Gl(V)$, let $H$ be the stabilizer in $G$
of a non-zero element of $V$, and let $H'$ be the stabilizer of a non-zero
vector in the dual space $\Hom_k(V,k)$. Then $H$ and $H'$ are Gassmann-Sunada
equivalent subgroups of
$\Gl(V)$ and $H$ is not conjugate to $H'$ if $d\ge 2$ and $(q,d)\ne(2,2)$.
Again, there is an automorphism $\alpha$ as in the previous remark, so $H$ and
$H'$ are order equivalent and also jump equivalent in $\Gl(V)$.
\end{exa}

One can obtain a Gassmann-Sunada triple that is not a jump triple out of this
example by enlarging the linear group to the group $V\rtimes G$, which is the
group
of affine linear transformations of~$V$. Then $V$ is a normal subgroup of
$V\rtimes G$, and $(V\rtimes G, V\rtimes H, V\rtimes H')$ is a Gassmann-Sunada triple
that is not reduced if $d>0$.

\begin{prop}
If $d\ge 2$ and $(q,d)\ne(2,2)$, then $(V\rtimes G, V\rtimes H , V\rtimes H' )$
is a Gassmann-Sunada triple that is not a jump triple.
\end{prop}

\begin{proof}
Let $S$ be the subset of those elements of $V\rtimes G$ that have a conjugate
in $\{0\}\rtimes H$. Since $H$ and $H'$ are Gassmann-Sunada equivalent in $G$ this
is also the subset of elements of $V\rtimes G$ that have a conjugate in
$\{0\}\rtimes H'$. We will show that
$
\langle S\cap (V\rtimes H)\rangle = V\rtimes H \text{ and }
\langle S\cap (V\rtimes H')\rangle \ne V\rtimes H'.
$
By taking $T=V\rtimes G$ in the defining property of jump equivalence
the Proposition will then follow.

Note that the vector space $V$ is a left module over the
group ring~$\Z[H]$. The augmentation ideal $I(H)$ in this ring is
the ideal generated by all elements $1-h$ with $h\in H$.
We now claim that
\[
\langle S \cap (V\rtimes H)\rangle = (I(H)\cdot V) \rtimes H.
\]
To see ``$\subset$'' note that any element of $S \cap (V\rtimes H)$ is of the
form
\[
(v,\gamma)(0,h)(v,\gamma)^{-1}=((1-\gamma h \gamma^{-1})v, \gamma h
\gamma^{-1}) \]
with $h\in H$, $\gamma\in G$ and $v\in V$ satisfying $\gamma h \gamma^{-1}\in
H$, and therefore $1-\gamma h \gamma^{-1}\in I(H)$.
For the other inclusion, note first that $\{0\}\times H \subset S \cap
(V\rtimes H)$.  By taking $\gamma=1$ in the identity above, one sees that for
every $h\in H$ and $v\in V$ we have $((1-h)v,h) \in S \cap (V\rtimes H)$, and
therefore $((1-h)v,1) \in S \cap (V\rtimes H)$. One then deduces ``$\supset$''.
This shows the claim, which by the same argument also holds with $H$
replaced by~$H'$. To finish the proof, it remains to show that $I(H)\cdot V= V$ and $I(H')\cdot V \ne
V$.  

To see that $I(H')\cdot V \ne V$,
we notice that since $H'$ is the stabilizer of a non-zero linear map $ \phi\colon\;V\to
k$, any element $h\in H'$ acts trivially on $V/\ker\phi$. 
It then follows immediately that $I(H') \cdot V \subset
\ker\phi\subsetneq V$.

To see that $I(H)\cdot V=V$, we
choose a basis $e_1,\ldots,e_d$ of $V$
as a vector space over $k$ so that $H$ fixes the first basis vector.
The element $h\in H$ sending $e_1$ to $e_1$ and $e_i$ to $e_i+e_{i-1}$
for $i>1$ satisfies  $(1-h)V=\operatorname{Span}_{k}\{e_1, \ldots , e_{d-1} \}$.
Thus, it now suffices to show that $e_d\in I(H)\cdot V$. When $d\ge 3$ this
follows by what we did already, because of the freedom to choose the basis.
When $d=2$, we assumed
that $q\ne 2$, so there is an element $\lambda\in k$ that is not $0$ or~$1$.
Now the element $h$ of $H$ given by $e_1\mapsto e_1$ and $e_2\mapsto \lambda
e_2$ satisfies $(1-h)V=ke_2$.
\end{proof}

Taking $q=2$ and $d=3$ this gives a non-reduced Gassmann-Sunada triple $(G,H,H')$
with $[G:H]=7$ that is not a jump triple. There are no non-trivial
Gassmann-Sunada triples where $[G:H]$ is smaller.

\begin{exa}[Todd-Komatsu method]\label{exa:GSMethod}
A different method to construct Gassmann-Sunada triples is to start with two
finite groups $H$ and $H'$ of some order $n$, such that for all divisors $d$ of
$n$
the groups $H$ and $H'$ have the same number of elements of order~$d$.  By
numbering the group elements, and considering the regular actions of the groups
on themselves by left-multiplication we obtain embeddings of $H$ and $H'$ into
$S_n$. Under such an embedding each element of order $d$ is sent to a product
of $n/d$ disjoint cycles of length~$d$. Then since two elements in $S_n$ are
conjugate exactly when they have the same cycle decomposition, it follows that
for each $d$ all elements of order $d$ from $H$ or $H'$ are conjugate in~$S_n$.
This shows that the triple $(S_n,H,H')$ is a Gassmann-Sunada triple.  This
triple is non-trivial (i.e., $H$ and $H'$ are non-conjugate subgroups of $S_n$)
if and only if $H$ is not isomorphic to~$H'$.

The Komatsu \cite{Komatsu} triples are the ones obtained in this way when
$H$ and $H'$ are two finite groups of the same order $n$ and
with the same odd prime $p$ as exponent; that is, $x^p =1$ for all $x \in H$ or
$H'$. In this case all non-trivial elements of $H$ and $H'$ are in the same
conjugacy class of $S_n$, so $H$ and $H'$ are order equivalent
and jump equivalent. Thus, as noted by Sormani and Wei \cite[Proposition
10.6]{SW} Komatsu isospectral manifolds have the same covering spectrum.

Another instance of this method  was used by Todd \cite{Todd} in 1949 to give a
counterexample to a conjecture of Littlewood that for any Gassmann-Sunada
triple $(G, H, H')$ the groups $H$ and $H'$ are isomorphic.
Consider the groups $H=C_8\times C_2$ and $H'=C_8\rtimes C_2$, where $C_n$
denotes a cyclic group of order $n$ and the generator of $C_2$ in the
semidirect product $H'$ acts on $C_8$ by raising all elements to the fifth
power. Then $H$ and $H'$ are groups of order $16$ that each have $1$, $3$, $4$
and $8$ elements of order $1$, $2$, $4$ and $8$ respectively, and we obtain
Todd's Gassmann-Sunada triple $(S_{16},H,H')$. Again, $H$ and $H'$ are also
order equivalent. However, Todd's idea can be slightly modified to produce
a Gassmann-Sunada triple that is not a jump triple.
\end{exa}

\begin{prop}\label{Prop:ECSExamples}
Let $N=C_4 \times C_2=\langle a, b: a^4=b^2=1, ab=ba\rangle$, and
define the groups $H=N\times \langle c: c^2=1\rangle$
and $H'=N\rtimes \langle c: c^2=1\rangle$ of order $16$, where in the
semidirect product $c$ acts on $N$ by $a\mapsto a$ and $b \mapsto
a^2b$. Then we have

\begin{enumerate}
\item for every $d\mid 16$ the groups $H$ and $H'$ have the
same number of elements of order $d$;
\item
the Gassmann-Sunada triple $(S_{16},H,H')$ is not a jump triple;
\item
the Gassmann-Sunada triple $(S_{64}, H\times C_4, H'\times C_4)$
is a jump triple, but the subgroups are not order equivalent.
\end{enumerate}
\end{prop}

\begin{proof}
Note that for the group $H$ the square of $a^ib^jc^k$,
where $i,j,k\in \Z$, is $a^{2i}$,
while for $H'$ the square of $a^ib^jc^k$ is~$a^{2(i+jk)}$.
Thus, in both groups it is true that given any
$j,k\in \{0,1\}$ there are exactly two $i\in \{0,1,2,3\}$
such that $(a^ib^jc^k)^2\ne 1$.
One deduces that $H$ and $H'$ both
have $8$ elements of order $4$ and $7$ elements of order~$2$. Since $H$ and
$H'$ are both of order $16$ this establishes the first statement.

By the method described in \ref{exa:GSMethod}, we obtain a Gassmann-Sunada
triple $(S_{16},H,H')$.  Now it is easy to see that the elements of order 2
generate a subgroup of index
2 in $H$, but in $H'$ they generate the whole group.
Taking $S=\{\si\in S_{16}:\si^2=1\}$ and $T=S_{16}$ we see that
$H$ and $H'$ are not jump equivalent in~$S_{16}$.

For the third statement, note that the two subgroups are contained in three
conjugacy classes $C_1$, $C_2$, $C_4$ of $S_{64}$, where $C_i$ consists of
elements of order~$i$.
We now have $\langle C_4\cap H\rangle =H$ and
$\langle C_4\cap H'\rangle =H'$, while
$[H: \langle C_2\cap H\rangle] =4$ and
$[H': \langle C_2\cap H'\rangle] =2$.
Thus, we see that $H$ and $H'$ are jump equivalent,
and Gassmann-Sunada equivalent in $S_{64}$, but not order equivalent.
\end{proof}

\begin{exa}[Mathieu groups]
Another example of a Gassmann-Sunada triple that is not a jump triple
is the triple $(M_{23}, 2^{4}A_{7}, M_{21}*2)$ mentioned by Guralnick and Wales
\cite[p. 101]{GW}. Here $M_i$ denotes the Mathieu group in degree~$i$. To see
that the jump condition does not hold, one may check, for instance
with Magma \cite{Magma}, that both subgroups $2^{4}A_{7}$ and $M_{21}*2$ are
generated by their elements of order $3$, and that $2^{4}A_{7}$ is generated by
its elements of order $2$ while $M_{21}*2$ is not.
\end{exa}

\begin{exa}[Order equivalence]
We conclude this section with a triple $(G,H,H')$ where $H$ is order equivalent
to $H'$, but not Gassmann-Sunada equivalent.  Consider the action of the
alternating group $A_4$ on a regular tetrahedron.
Let $E$ be the set of its $6$ edges and choose an element $e\in E$.
Then the group $A_4$ acts transitively on $E$, and the stabilizer of $e$ is a
subgroup of order $2$.

Let $V$ be the vector space over the field of 2 elements with $E$
as a basis. Then $A_4$ acts on $V$ by permuting coordinates.
We now let $W$ be the $5$-dimensional subspace of $V$ where the sum of the
coordinates is zero, and we put $G=W\rtimes A_4$, which is a group of
order~$384$.  We let $H$ be the 4-dimensional subspace of $W$ consisting of
the vectors with coordinate $0$ at~$e$.  To define $H'$, choose a subset $E'$
of $E$ consisting of $3$ edges in such a way that $E'$ contains two
non-adjacent edges, and let $H'$ be the 4-dimensional subspace of $W$
consisting of the
vectors whose coordinate sum over $E'$ is zero. Then one can check that
$H$ and $H'$ are order equivalent in $G$, but not Gassmann-Sunada equivalent.
\end{exa}


\begin{exa}[Additional jump equivalent triples that are not Gassmann-Sunada equivalent]
In Remark~\ref{rem:PresIndex} we gave an example of a jump triple that is not a Gassmann-Sunada triple. 
We now demonstrate that such examples can be constructed in a rather routine fashion. Indeed, 
let $V$ be a finite dimensional vector space over a finite field $k$, and let 
$G = V \rtimes \Gl(V)$
denote the affine transformation group of $V$. We will agree to call 
an element $v \in V \subset G$ a \emph{translation} and any subspace $W \subset V$ a \emph{translation subgroup}. 
Since $\Gl(V)$ acts transitively on the non-trival vectors of $V$ it follows that the conjugacy 
class of an element $v \in V-\{0\}$ consists entirely of all the non-trivial translations of $V$. 
Therefore, any two non-trivial translation subgroups $V_1, V_2 \leq G$ are 
Kronecker equivalent in $G$ (cf. \cite[Thm. 2.10(i)]{LMNR}).

Now, suppose $S \subset G$ is stable under conjugation. Then, we have $\langle V_i \cap S \rangle$ is trivial or all of  $V_i$ (for $i = 1,2$) depending on whether $S$ contains a non-trivial translation or not, and we conclude that $(G, V_1, V_2)$ is a jump triple. If we take $V_1$ and $V_2$ to be of different dimensions (and hence of different orders), then $(G, V_1, V_2)$ is a jump triple that is not Gassmann-Sunada.
\end{exa}

\bigskip


\section{Covering spectra of isospectral manifolds}\label{Sec:Isospectral}

In this last section, we turn to isospectral manifolds and their covering
spectra. We first describe how to use the results of the previous two
sections in order to prove Theorem \ref{thm:distinct}.
We then show that the isospectral lattices of Conway and Sloane can also
give rise to isospectral flat tori of dimension 4 with distinct
covering spectra. At the end of the section we 
study the covering spectra of isospectral Heisenberg manifolds 
and refute the claim made in \cite[Example 10.3]{SW}.

\begin{proof}[Proof of Theorem \ref{thm:distinct}]
Given $n\ge 3$ we will produce isospectral Riemannian manifolds of dimension
$n$ with distinct covering spectra. In the previous section we saw that there
exists a Gassmann-Sunada triple $(G,H,H')$ such that $H$ and $H'$ are not
jump-equivalent in~$G$. Suppose that $G$ can be generated by $k$ generators:
$g_1, \ldots , g_k$. For instance, $k=2$ for the triple in part $(2)$ of
Proposition
\ref{Prop:ECSExamples}. 
We will now construct a closed manifold $M$ of dimension $n$ on which $G$ acts
freely as in \cite[Sec. 11.4]{Buser}.

Let $N_0$ be any compact orientable surface of genus at least $k$. 
Then $\pi_1(N_0)$ has a quotient group that is free on $k$ generators.
One can see this with an explicit description of $\pi_1(N_0)$
as in \cite[3.8.12]{Spanier}, or by looking through the 
$k$ holes and projecting the surface to a closed disc in the plane
with $k$ points missing.
Since $G$ is a quotient of the free group on $k$ generators,
there is a surjective homomorphism $\phi:
\pi_1(N_0) \to G$.
(One can also obtain an explicit homomorphism of $\pi_1(N_0)$ onto $G$ as in
\cite[Sec. 11.4]{Buser}.) Now let $N$ be the regular covering of $N_0$
corresponding to the normal subgroup
$\ker(\phi)$ of $\pi_1(N_0)$.  Then the group of deck transformations
of $N$ over $N_0$ is isomorphic to $G$; see \cite[Thm.~2.3.12, 2.5.13, Cor.~2.6.3]{Spanier}. Letting $M = N \times Y$ be the product of $N$
with any closed manifold $Y$ of dimension $n-2$, then we see that $M$ is a
closed $n$-manifold on which $G$ acts freely. Then, since $(G, H, H')$ is not a
jump triple, Theorem~\ref{thm:CovSunada} tells us that $M$ admits a
$G$-invariant metric such that the quotient manifolds $H \backslash M$ and $H'
\backslash M$ have distinct covering spectra, while Sunada's result says they
have the same Laplace spectra.
\end{proof}

\begin{exa}[Isospectral flat tori with distinct covering spectra]\label{exa:tori}
An alternative way to obtain examples of isospectral manifolds with distinct
covering spectra of dimension $4$ is by using the construction of rank 4
isospectral lattices by Conway and Sloane  \cite{CS}.
It was pointed out by Gordon and Kappeler \cite[Sec. 4.4]{GK} that these
lattices arise from a Gassmann-Sunada triple. We briefly explain the
construction.

Let $A$ be the additive group of the group ring $\Z[V_4]$ where
$V_4=\{1,\sigma,\tau,\rho\}$ is the Klein group of order $4$ with unit
element~$1$.  Then let $H=\langle 3A, \sigma+\tau+\rho, 1+\rho-\tau\rangle$
and $H'=\langle 3A, 1+\sigma+\tau, 1+\rho-\tau\rangle$, which are both subgroups of $A$. We view $H$ and $H'$ as subgroups of the semidirect product $G = A\rtimes V_4$.

Note that the group ring $\R[V_4]$ over the field of real numbers has a basis
as a vector space over $\R$ consisting of the primitive idempotents
$e_1=(1+\sigma+\tau+\rho)/4$, $e_2= (1+\sigma-\tau-\rho)/4$,
$e_3=(1-\sigma+\tau-\rho)/4$, $e_4=(1-\sigma-\tau+\rho)/4$.
For each element $g$ of $V_4$ the map
$\R[V_4]\to \R[V_4]$ sending $x$ to $gx$ is given by a diagonal matrix.
We now take any Euclidean structure
on $\R[V_4]$ so that these eigenspaces are orthogonal. 
Then $G$ acts by isometries on~$\R[V_4]$.

A note on notation: our lattice $H$ is denoted $\mL$ in \cite{GK} and $\mL^+$
in \cite{CS}, while $H'$ is $\mL'$ in \cite{GK}, and reflecting across the
hyperplane spanned by $e_1, e_2$ and $e_3$ gives $\mL^-$ in \cite{CS}.

Now consider the flat torus $\R[V_4]/3A$ as a Riemannian manifold. The group
$G/3A$ acts on $\R[V_4]/3A$, and its quotients by the subgroups  $H/3A$ and $H'/3A$ are the flat tori
$T=\R[V_4]/H$ and $T'=\R[V_4]/H'$. It is not hard to see (cf. \cite[Sec. 4.4]{GK})
that $(G/3A, H/3A,H'/3A)$ is a Gassmann-Sunada triple.  Note that $G/3A$ does
not act freely on $\R[V_4]/3A$, but $H/3A$ and $H'/3A$ do.  In this context, 
Sunada's result is still valid \cite{Berard1, Berard2}, so $T$ and $T'$ are
isospectral.

For a judicious choice of our Euclidean structure on $\R[V_4]$ we now obtain
distinct covering spectra.  Specifically, the following table gives three
examples of Euclidean structures and the covering spectra they give rise to.

\medskip
\begin{center}
\begin{tabular}{c  c  c}
$3\langle e_1,e_1\rangle,\ldots, 3\langle e_4,e_4\rangle $&
$\CovSpec(T) $&
$\CovSpec(T')$\\
\hline
\hline
$1,4,10,13$ &
$\{\sqrt{3}, \sqrt{5},
\sqrt{6}, \sqrt{7} \}$&
$\{\sqrt{3}, \sqrt{5},
\sqrt{7} \}$
\\

$2, 8, 14, 20$ &
$\{\sqrt{5}, 3,
\sqrt{10}, \sqrt{13} \}$&
$\{ \sqrt{5}, 3,
\sqrt{12}, \sqrt{13} \}$
\\

$1,7,13,19$ &
$\{2, \sqrt{8}, \sqrt{10}\}$&
$\{2, \sqrt{8}, \sqrt{10}\}$
\\
\end{tabular}
\end{center}
\medskip\noindent
In the last row the covering spectra are the same, but the multiplicities
in the sense of Example \ref{Exa:FlatTori} or \cite[Definition 6.1]{SW}
are not: they are $1,2,1$ (respectively) for $T$ and $1,1,2$
(respectively) for~$T'$. This table was obtained using
Magma \cite{Magma} and can be verified (tediously) by hand.
\end{exa}

\begin{rem}\label{rem:tori}
In fact, $H/3A$ and $H'/3A'$ are jump equivalent in $G/3A$, but since $\R[V_4]/3A$
is not simply connected, this does not contradict Theorem \ref{thm:CovSunada}.
Furthermore, $H/3A$ and $H'/3A'$ are order equivalent in $G/3A$.
However, one can check that $H/9A$ and $H'/9A$ are \emph{not} jump equivalent
in $G/9A$.
\end{rem}

\begin{exa}[Heisenberg Manifolds]\label{Exa:Heis}
Let $(V, \inner)$ be a real inner product space of dimension $2n$
endowed with a non-degenerate symplectic form $\omega$; that is,
$\omega\colon\; V\times V \to \R$ is a non-degenerate skew-symmetric form.
The Heisenberg group $H(V)$ associated to $V$ is the manifold
$H= V\times \R$, together with the group operation
$$(x,t)(x',t') = (x + x', t + t' + \omega(x,x')/2).$$
With this operation $H$ becomes a non-commutative real Lie group.
The identity element $e\in H$ is the pair $(0,0)$.
The center $Z=\{0\}\times \R$ of $H$ is equal to the commutator
subgroup $[H,H]$ of $H$, so $H$ is $2$-step nilpotent, and the
commutator pairing $H\times H\to [H,H]$ is given by the
symplectic form: $$(x,t)(x',t')(x,t)^{-1}(x',t')^{-1}=(0,\omega(x,x')).$$

The tangent space of $H$ at $e$ is $V\times \R$, which we view
as an inner product space by putting the standard inner product on $\R,$
and taking the orthogonal direct sum with $(V,\inner)$.  We extend this inner
product uniquely to a left-invariant Riemannian metric $g$ on $H$ . Note that
the
group isomorphism $V\to H/Z$ is an isometry between the Euclidean
space $V$ with its flat metric and the quotient manifold $H/Z$ together
with its induced structure of a Riemannian manifold.

Now suppose we have a lattice $\mL$ of full rank in $V$, and a positive real
number $c$ such that $\omega(\mL,\mL)\subset c\Z$. Then $\Gamma=\mL\times c\Z$
is a discrete
subgroup of~$H$ that acts freely and by isometries on $H$ by left
multiplication. The quotient space $X=\Gamma\backslash H$ is then a Riemannian
manifold, called a \emph{Riemannian Heisenberg manifold}.

Note that $Z/(\{0\}\times c\Z)$ acts freely by isometries on $X$, and that
the quotient $X/Z$, considered as a Riemannian manifold, is just the flat torus
$T=V/\mL$.  The covering spectrum of $T$ is the jump set of the length map
$m_\mL$ on~$\mL$ as
in Example \ref{Exa:FlatTori}.

We will now show how to obtain the covering spectrum of $X$ from that of~$T$.
The covering spectrum of $X$ is the jump set of the length map
$m_\Gamma(\gamma)=\inf_{h\in H} d_H(h,\gamma h)$ on $\Gamma$,
where $d_H$ is the geodesic distance on~$H$ (see Example~\ref{exa:MinMark}).

\begin{prop}\label{prop:HeisCovSpec}
Let $\delta_T=\inf \CovSpec(T)$, and put $\delta_Z=m_\Gamma((0,c))$.
Then we have
$$
\CovSpec(X)=
\begin{cases}
\{\delta_Z\}\cup \CovSpec(T) &
\text{if $\delta_Z < \delta_T$;}\\
\CovSpec(T) &
\text{otherwise.}
\end{cases}
$$
\end{prop}

\begin{proof}

It is well-known for Riemannian Heisenberg manifolds that
\begin{enumerate}
\item $m_\Gamma(v,kc) = m_{\mL}(v)$ for all non-zero $v\in \mL$,
and
\item among all free homotopy classes represented by central elements
$\{0\}\times c\Z$ of $\Gamma$, the one that can be represented by a closed
geodesic of smallest length is
$(0,c).$
\end{enumerate}
For details of (1) and (2) and/or a precise calculation of the geodesics and
$\delta_Z$ we refer the reader to \cite{Kaplan}, \cite{Gordon}, or
\cite{Eberlein}.

To deduce the proposition, suppose first that $\delta_Z<\delta_T$.
Then by property (2) the element $(0,c)$ has the smallest non-zero length in~$\Gamma$.  The
intermediate groups of $\langle (0,c)\rangle \subset \Gamma$ now correspond
exactly to
the subgroups of $\mL$, and by property (1) the filtrations induced by
$m_{\Gamma}$ and $m_{\mL}$ (respectively) correspond as well, so this gives us
the first case.

For the second case, suppose that $v\in \mL$ with $v\ne 0$ and
$\delta_T(v) \le \delta_Z$. Then we have $(0,kc)=(v,kc)(v,0)^{-1}$,
and by property (1) both $(v,0)$ and $(v,kc)$ have length~$\delta_T(v)$.
Thus, $\langle (0,c), (v,0) \rangle$ is contained in the
smallest non-trivial group in the filtration of $\Gamma$,
and the second case follows.
\end{proof}

\begin{rem}\label{rem:HeisEquiv}
By \cite[Prop 2.16]{Gordon3} any quotient of a $2n+1$-dimensional
simply connected Heisenberg Lie group endowed with a left-invariant metric by
a cocompact, discrete subgroup is isometric to a Riemannian Heisenberg manifold
as described above, i.e., to one of the form $\Gamma\backslash H(V)$,
where $(V,\omega,\inner)$ and $\Gamma=\mL \times c\Z$ are as above.

Moreover, if two such Riemannian Heisenberg manifolds are isospectral,
then they are isometric to manifolds
$X=(\mL\times c\Z)\backslash H(V)$ and $X'=(\mL'\times c\Z) \backslash H(V)$
given with the same $V$, $\langle \cdot,\cdot\rangle$, $\omega$, $c$, 
and two cocompact lattices $\mL$ and $\mL'$ in $V$ such that
the flat tori $T=V/\mL$ and $T'=V/\mL'$ are isospectral. 
In fact, two Heisenberg manifolds $X$ and $X'$ as above
are isospectral if and only if $T=V/\mL$ and $T'=V/\mL'$ are isospectral.
To see these comments, the reader may consult \cite[Prop. 2.14, 2.16]{Gordon3} and
\cite[Prop. III.5]{Pesce2}.
\end{rem}

Combining the remark with the proposition we have the following.

\begin{cor}\label{cor:HeisCovSpec}
Suppose that $X$ and $X'$ are isospectral Heisenberg manifolds as 
in Remark \ref{rem:HeisEquiv}.
Then $X$ and $X'$ have the same covering spectra if and only if $T$ and $T'$
have the same covering spectra.
\end{cor}

\begin{proof}
By statement (2) in the proof of Proposition \ref{prop:HeisCovSpec} we see
that $\delta_Z=\delta_Z',$ since both are equal to the distance in $H(V)$ from
$(0,0)$ to $(0,c)$.  As the flat tori $T=V/\mL$ and $T'=V/\mL'$ are
isospectral, they have the same length spectrum, so there is a one to one
correspondence between the lengths of the vectors in $\mL$ and the lengths of
the vectors in $\mL'$. In particular, we must have
$\delta_T=\delta_{T'}$. The result now follows from
Proposition \ref{prop:HeisCovSpec}.
\end{proof}

\begin{exa}[Isospectral Heisenberg manifolds with different covering spectra]
Let $\mL, \mL'\subset \R[V_4]$ be a pair of lattices of Conway and Sloane
as in Example~\ref{exa:tori}, that give rise to 
isospectral 4-dimensional flat tori with distinct
covering spectra. 
One easily checks that $\Gamma=\mL\times\Z$ and $\Gamma'=\mL'\times\Z$ are
discrete subgroups of $H(V),$ where $V=\R[V_4]=\R^{4}$ with the standard symplectic
form. By Remark~\ref{rem:HeisEquiv}, $X=\Gamma\backslash H(V)$ and
$X'=\Gamma'\backslash H(V)$ are isospectral, since $T=V/\mL$ and $T'=V/\mL'$
are isospectral. By Proposition~\ref{prop:HeisCovSpec}, they have distinct
covering spectra, as $T$ and $T'$ have distinct covering spectra.
\end{exa}

\begin{rem}
It follows from Corollary~\ref{cor:HeisCovSpec} that the example of isospectral
Heisenberg manifolds with different covering spectra given by Sormani and Wei
(see \cite[Example 10.3]{SW}, \cite{SWerrata}) is incorrect, as their tori
$T$ and $T'$ have the same covering spectrum.
\end{rem}

\begin{rem} In \cite[Definition 6.1]{SW}, Sormani and Wei assign a {\it basis
multiplicity} to each element $\delta$ of the covering spectrum of a 
compact Riemannian manifold $M$. It is the minimal number of generators of
length $2 \delta$ that one needs to add to $\Fil^\delta\fund$ to generate the
next bigger group in the filtration.
One can show
that if $X$ is a Riemannian Heisenberg manifold with $\delta_Z \geq \delta_T$,
then the basis multiplicity of $\delta_T$ in $\CovSpec(X)$ is either equal to
the basis multiplicity of $\delta_T$ in $\CovSpec(T)$ or equal to one more than
the basis multiplicity of $\delta_T$ in $\CovSpec(T).$  All other basis
multiplicities of $\CovSpec(X)$ are unaffected; i.e., their basis multiplicity
in $\CovSpec(X)$ is equal to their basis multiplicity in $\CovSpec(T).$ The
isospectral Heisenberg manifolds of \cite[Example 2.4a]{Gordon}
have the same covering spectra by Corollary~\ref{cor:HeisCovSpec}, but
$\delta_Z$ has a different basis multiplicity in the two examples.  
\end{rem}

\end{exa}



\bibliographystyle{amsalpha}


\end{document}